\documentclass[final,onefignum,onetabnum,reqno]{siamart190516}
\usepackage{hyperref}
\usepackage{amsfonts,graphicx,mathrsfs,bm,indentfirst}
\usepackage{amsmath,amssymb,amsfonts,subfigure,subeqnarray,latexsym,fancybox,comment,amssymb}
\usepackage{epsfig,color}
\usepackage{multirow}

\newtheorem{assumption}[theorem]{Assumption}

\newcommand{\thmref}[1]{Theorem~\ref{#1}}

\newcommand{\secref}[1]{Section~\ref{#1}}
\newcommand{\lemref}[1]{Lemma~\ref{#1}}
\newcommand{\lemrefs}[2]{Lemmas~\ref{#1} and ~\ref{#2}}
\newcommand{\propref}[1]{Proposition~\ref{#1}}
\newcommand{\proprefs}[2]{Propositions~\ref{#1} and ~\ref{#2}}

\newcommand{\assref}[1]{Assumption~\ref{#1}}

\newcommand{\figref}[1]{Figure~\ref{#1}}

\newcommand{\tabref}[1]{Table~\ref{#1}}

\newcommand{\eq}[1]{\begin{eqnarray}\label{#1}}
\newcommand{\qe}{\end{eqnarray}}

\newcommand{\be}{\begin{eqnarray}}
\newcommand{\ee}{\end{eqnarray}}
\newcommand{\bal}{\begin{aligned}}
\newcommand{\eal}{\end{aligned}}
\newcommand{\bes}{\begin{eqnarray*}}
\newcommand{\ees}{\end{eqnarray*}}
\newcommand{\bs}{\begin{subeqnarray}}
\newcommand{\es}{\end{subeqnarray}}
\newcommand{\bss}{\begin{subeqnarray*}}
\newcommand{\ess}{\end{subeqnarray*}}

\newcounter{saveeqn}

\def\hat{\widehat}

\def\diam{\operatorname{diam}}

\def\mcF{\mathcal F}
\def\Oh{\mathcal O}

\def\mbbR{\mathbb R}

\def\tt{T}

\def\p{\partial}

\def\O{\Omega}

\def\NChz{{\mathcal NC}^h_0}
\def\NChz2d{{[\mathcal{NC}}^h_0]^2}
\def\tNChz2d{\widetilde {[\mathcal{NC}}^h_0]^2}
\def\oa{\overline{a}}
\def\ou{\overline{u}}

\def\vep{\varepsilon}

\def\G{\Gamma}

\def\and{\quad\text{and}\quad}

\def\tH{\widetilde H}

\def\<{\left\langle}
\def\>{\right\rangle}

\def\mbA{\mathbf A}
\def\mbJ{\mathbf J}
\def\mbW{\mathbf W}

\def\be{\mathbf e}

\def\b1{\mathbf 1}

\def\Tau{{\mathcal T}}

\def\bnu{{\boldsymbol \nu}}

\def\grad{\nabla}
\def\div{\nabla\cdot}

\def\mbe{\mathbf e}
\def\mbx{\mathbf x}
\def\mby{\mathbf y}
\def\dx{\kern0.1ex\operatorname{d\mbx}}
\def\dy{\kern0.1ex\operatorname{d\mby}}
\def\ds{\kern0.1ex\operatorname{ds}}
\def\dsig{\kern0.1ex \operatorname{d}\kern-0.5ex\boldsymbol{\sigma}}
\def\dhsig{\kern0.1ex \operatorname{d}\kern-0.5ex\hat{\boldsymbol{\sigma}}}
\def\dhatx{\kern0.1ex\operatorname{d\hat{\mbx}}}
\def\dhats{\kern0.1ex\operatorname{d\hat{s}}}

\def\forany{\quad\forall}

\newcommand{\vertiii}[1]{{\left\vert\kern-0.25ex\left\vert\kern-0.25ex\left\vert #1 \right\vert\kern-0.25ex\right\vert\kern-0.25ex\right\vert}}
\newcommand{\brknmH}[1]{{\vertiii {#1}}_H}
\newcommand{\brknmh}[1]{{\vertiii {#1}}_{h,K_{\delta,i}}}
\newcommand*{\avint}{\mathop{\ooalign{$\int$\cr$-$}}}

\newsiamthm{claim}{Claim}
\newsiamremark{remark}{Remark}
\newsiamremark{hypothesis}{Hypothesis}
\crefname{hypothesis}{Hypothesis}{Hypotheses}
\begin{document}
\markboth{$P_1$--nonconforming quadrilateral element for HMM}{Yim, Sheen \& Sim}
\title{$P_1$--Nonconforming Quadrilateral Finite Element Space with
  Periodic Boundary Conditions:\\
Part II. Application to the
  Nonconforming Heterogeneous Multiscale Method}
\author{Jaeryun Yim\thanks{
215 Bongeunsa-ro, Gangnam-gu, Seoul 06109, Korea.
\email{jaeryun.yim@gmail.com}}
\and
Dongwoo Sheen
\thanks{Department of Mathematics, Seoul National University, Seoul 08826, Korea.
  \email{sheen@snu.ac.kr}}
\and Imbo Sim\thanks{\small Department of Mechanical Engineering, Dong-A University, 37 Nakdong-daero 550
  beong-gil, Saha-gu, Busan 49315, South
   Korea.
\email{imbosim@dau.ac.kr}}}

\maketitle

\allowdisplaybreaks

\newif\iflong
\longfalse

\begin{abstract}
A homogenization approach is one of effective strategies to solve
multiscale elliptic problems approximately. The finite element
heterogeneous multiscale method (FEHMM) which is based on the finite
element makes possible to simulate such process numerically. 
In this paper we introduce a FEHMM scheme for multiscale elliptic
problems based on nonconforming spaces. In particular we use the
noconforming element with the periodic boundary condition introduced
in the companion paper.
Theoretical analysis derives a priori error estimates in the standard Sobolev norms. Several numerical results which confirm our analysis are provided.
\end{abstract}

\begin{keywords} Heterogeneous multiscale method; nonconforming finite
  element method; homogenization
\end{keywords}

\begin{AMS}
  65N30
\end{AMS}

\section{Introduction}
The finite element method (FEM) is one of the successful methods to
approximate
the solution of partial differential equations derived in various fields of studies.
However it has a drawback when we treat a problem containing coefficient tensors with heterogeneity.
For instance, when the coefficient is highly oscillatory in micro scale, it is necessary to consider a sufficiently refined mesh
consisting of elements comparable with the micro scale
in order to get a numerical solution which is sufficiently close to the exact solution.
Such a refinement increases the number of unknowns in the system of
equations we set to solve, which is a critical burden for
  numerical simulation.
In these decades, several efficient multiscale methods have been
  developed to overcome this shortage of the standard FEM. Among them,
  we refer to numerical homogenization \cite{babuska1976homogenization, babuska1994special, efendiev2004numerical, engquist2008asymptotic}, 
MsFEM (multiscale finite element methods) and GMsFEM (generalized MsFEM)
\cite{efendiev2009multiscale, efendiev2013generalized, efendiev2014generalized, efendiev2015spectral, hou1997multiscale,
  hou1999convergence}, VMS (variational multiscale finite element methods) \cite{hughes1998variational},
MsFVM (multiscale finite volume methods) \cite{jenny2003multiscale} and HMM (heterogeneous multiscale
methods) \cite{abdulle2012heterogeneous, e2003heterogeneous}.

Most of the works cited above are based on the conforming finite
element framework, while nonconforming finite elements have
prominence for their numerical stability 
especially in solving incompressible fluid flows 
\cite{cdssy, cdy, crouzeix1973conforming, kim-yim-sheen, rannacher-turek, turek},
nearly incompressible elastic materials \cite{brenner-sung-elasticity,
  lls-nc-elast}, and biharmonic problems \cite{morley, Hu-Shi2009,
  Hu-Shi-Xu2011, mao-shi-high-morley, park-sheen-morley,
  shi-xie-morley-aniso}. 
We remark that the first nonconforming analysis in multiscale methods
has been done by Efendiev {\it et al.} in \cite{efendiev2000convergence} in order to analyze
the nonconforming nature where conforming finite elements were adopted
with oversampling. Recently Lozinski {\it et al.} began to develop
the idea of using the Crouzeix--Raviart element \cite{crouzeix1973conforming} in multiscale
finite element methods intensively for elliptic and Stokes
problems \cite{bris2013msfem, lozinski2013methode, bris2014msfem}. Especially it has been shown that in applying the
Crouzeix--Raviart nonconforming
multiscale finite element for perforated domains the
use of carefully--chosen extra bubble elements improves stability and efficiency
substantially \cite{muljadi2015nonconforming, degond2015crouzeixraviart}.
In the generalized MsFEM direction, Lee and Sheen
\cite{lee-sheen-ncgmsfem} introduced the general
framework for applying the nonconforming concept to elliptic
problems.

In the introduction of the FEHMM (Finite Element HMM), the standard
(conforming) finite element spaces have been employed for the micro and macro function
space (see, for instance, \cite{abdulle2009finite}, and the references therein), various types of finite element
spaces being adopted since then.
The discontinuous Galerkin (DG) approach is introduced to impose the
continuity of numerically homogenized jump across adjacent faces \cite{abdulle2012discontinuous}.
A posteriori error estimates and corresponding adaptive approaches for
the FEHMM are proposed \cite{abdulle2011adaptive}.
The generalized FEHMM approach by introduction of the on-line and
off-line space \cite{abdulle2012reduced}.


The aim of this
paper is to attempt to apply the nonconforming approach to the
primitive HMM for elliptic problems
rigorously, leaving the extension of the nonconforming approach as future works
to generalized HMMs and to more practical fluid and solid
mechanical problems, where the numerical stability is easily
achievable owing to the nature of nonconformity of elements.
In particular, in this paper we employ the 
$P_1$--NC quadrilateral element \cite{park-sheen-p1quad,
 cpark-thesis}, which will share the same nature as the lowest--order 
Crouzeix--Raviart element since it contains only linear polynomials in
each rectangles or hexahedrons. 
We would like to highlight that the use of rectangular--type element
has an extra advantage over the simplicial elements in the application
of multiscale methods, not to mention the simplicity of constructing
elements in three dimension, but the suitability of rectangular
type elements in dealing with periodicity.
In addition, by employing the $P_1$--NC quadrilateral element
  we do not need any additional linearization step for macro constraint
  functions
  in solving micro problems on sampling domains, since
  the $P_1$--NC quadrilateral element is already linear.
  This is another main advantage of our approach over using
 the conforming bilinear element. See Remark \ref{rem:no-linearization}.

This paper is organized as follows. In \secref{sec:preliminary} we
state in brief preliminaries and notations to be used later.
We introduce a nonconforming finite element heterogeneous multiscale method based on nonconforming spaces in \secref{sec:fehmm-nc}.
\secref{sec:analysis} is devoted to prove the main theorems for a
priori error estimates of our proposed method.
We give several numerical examples and results in \secref{sec:numerical-result}.

\section{Preliminaries}\label{sec:preliminary}

Let $\O \subset \mbbR^d, d =2,3,$ be a bounded domain with $C^\infty$ boundary
$\p \O$. Denote $\mbx=(x_1, \cdots, x_d)^t \in \mbbR^d$. Consider a multiscale elliptic problem
\begin{subequations}\label{eq:multiscale-elliptic}
\begin{align}
-\div \Big( \mbA^{\vep}({\bf x}) \grad u^{\vep}({\bf x}) \Big) &= f({\bf x})
\quad \mbox{in } \O, \\
u^{\vep} &= 0 \quad \quad \ \mbox{on } \p \O,
\end{align}
\end{subequations}
where $\vep \ll 1$ is a scale parameter. Here, the coefficient tensor $\mbA^\vep \in
[L^\infty(\O)]^{d\times d}$ is assumed to be symmetric, uniformly elliptic and bounded, {\em i.e.}, there exist
$\lambda>0$ and $\Lambda>0$ which do not depend on ${\bf x}$ such that
$\lambda |\xi|^2  \le \xi^t \mbA^{\vep}({\bf x}) \ \xi\le \Lambda
|\xi|^2$
for all $\xi \in \mbbR^d$.

%
\newpage
\subsection{Homogenization}
Denote by $Y = \prod_{k=1}^d [0,\ell_k]$ be a period cell for fixed
 $0<\ell_k,k=1,\cdots,d.$
Let $\mbA(\cdot,\cdot):\mbbR^d \times \mbbR^d \rightarrow
\mbbR^{d\times d}$ be a 
$Y$-periodic function with respect to the second variable, {\it i.e.},
$\mbA({\mbx},{\mby}) = \mbA({\mbx},{\mby} + \ell_k
{\mbe}_k)$ for $1\le k \le d,$
where $\mbe_j$ denotes the $j$--th standard unit vector in
$\mbbR^d.$
Assume that $\mbA^\vep({\mbx}) := \mbA({\mbx},{\mbx}/\vep).$
Then, the following result is well--known
  \cite{cioranescu1999introduction, jikov, sanchez}.
\begin{theorem}[Periodic case]
Suppose that $\mbA^\vep({\mbx}) := \mbA({\mbx},{\mbx}/\vep)$ where $\mbA({\mbx},{\mby})$ is $Y$-periodic for the variable ${\mby}=(y_1,\cdots,y_d)$. 
Let $f\in L^{2}(\O)$. Then there exists a homogenized coefficient tensor
$\mbA^0$ such that
\begin{equation*}
\left\{
\begin{aligned}
& u^{\vep} \rightharpoonup u^0 \mbox{ weakly in } H^1_0(\O), \\
& \mbA^{\vep} \grad u^{\vep} \rightharpoonup \mbA^0 \grad u^0 \mbox{ weakly in }
[L^2(\O)]^d, 
\end{aligned}
\right.
\end{equation*}
where $u^0$ is a unique solution in $H^1_0(\O)$ of the homogenized
problem:
\begin{equation}\label{eq:u0}
\left\{
\begin{aligned}
-\div \Big( \mbA^0({\mbx}) \grad u^0({\mbx}) \Big) &= f({\mbx})
\quad \mbox{in } \O, \\
u^0 &= 0 \quad \quad \ \mbox{on } \p \O.
\end{aligned}
\right.
\end{equation}
In fact, the homogenized coefficient $\mbA^0 = (A^0_{ij})$ is given by
\begin{equation*}
A^0_{ij}({\mbx}) = \frac{1}{|Y|} \int_{Y} \Big(A_{ij}({\mbx},{\mby}) + \sum_{k=1}^{d} A_{ik}({\mbx},{\mby})
\frac{\p \chi^j}{\p y_k} \Big) \dy,
\end{equation*}
where $|Y|$ denotes the volume of $Y$, and $\chi^j = \chi^j({\mbx},{\mby})$ the solution of the cell problem:
\begin{equation*}
\left\{
\begin{aligned}
& -\nabla_{\mby} \cdot \left(\mbA({\mbx},{\mby}) \grad_{\mby} \chi^j \right) 
= 
	\grad_{\mby} \cdot \left(\mbA({\mbx},{\mby}) \ {\mbe}_j \right) \quad \mbox{in } Y, \\
& \chi^j \mbox{ is $Y$-periodic}, \\
& \int_Y \chi^j \dy=0.
\end{aligned}
\right.
\end{equation*}
\end{theorem}

\subsection{Some notations}
Let $D$ be a bounded open domain in $\mbbR^d, d=2,3$.
Denote by $L^2(D),$ $H^1(D),$ and $H^1_0(D)$ the standard Sobolev spaces on $D$ with the standard Sobolev norms $\|\cdot\|_{0,D}$, $\|\cdot\|_{1,D}$, and (semi-)norm $|\cdot|_{1,D}$, respectively. 
By $C^\infty_{\#}(D)$ designate the set of $Y$--periodic $C^\infty$ functions
on $D$ and
by $H^1_{\#}(D)$ the closure of $C^\infty_{\#}(D)$ with respect to the norm $\|\cdot\|_{1,D}$ in $H^1(D)$. 
$\tH^1_{\#}(D)$ is a subspace of $H^1_{\#}(D)$ which consists of
functions whose mean values on $D$ are zero. We will mean by
$(\cdot,\cdot)_{D}$ the $L^2(D)$ inner product. In
the case of $D=\O$, the subscript $D$ on notations of norms and inner
product is omitted. For $d-1$ dimensional face $\tau$,
$\<\cdot,\cdot\>_\tau$ indicates the $L^2(\tau)$ inner product.

By $|D|$ we denote the volume of the domain $D$.
For an integrable function $v\in L^1(D)$, the mean value on $D$ is denoted by
$\avint_D v=\frac1{|D|}\int_D v.$
Throughout this paper $C$ denotes a generic constant and its value varies depending on the position where it appears.

Consider a family of triangulations $\Tau_h=(\Tau_h(D))_{0<h<1}$ for the domain $D$ consisting of quadrilateral elements.
Let $\mcF_h^i, \mcF_h^b,$ and $\mcF_h^{b,opp}$ denote the sets of all interior faces, of all boundary faces, 
and of all pairs consisting of two boundary faces on opposite position
of each period cell $Y$ in $\Tau_h$, respectively.
Set $\mcF_h=\mcF_h^i\cup \mcF_h^b.$

Set
\begin{align*}
V_{h}^{P_1}(D) &= \Big\{ v \in L^2(D) ~\Big|~ \left.v\right|_{K} \in
    \mathcal{P}_1(K) \forany K \in \Tau_h(D), \ 
    \<[v]_\tau,1\>_\tau = 0 \forany \tau\in\mcF_h^i \Big\}, \\
V_{h,0}^{P_1}(D) &= \Big\{ v \in V_{h}^{P_1}(D) ~\Big|~ 
	\< v,1\>_\tau = 0 \forany \tau\in\mcF_h^b \Big\}, \\
V_{h,\#}^{P_1}(D) &= \Big\{ v \in V_{h}^{P_1}(D) ~\Big|~ 
	\<v,1\>_{\tau_1}  = \<v,1\>_{\tau_2}  \forany (\tau_1,\tau_2) \in\mcF_h^{b,opp}, \ (v,1)_{D} = 0 \Big\},
\end{align*}
where $\mathcal{P}_1(K)$ denotes the set of linear polynomials on $K$, and $[\cdot]_\tau$ the jump across face $\tau$.
Let $|\cdot|_{1,h,D}$ denote the mesh-dependent broken energy
norm on $V_{h}^{P_1}(D)$, {\it i.e.,}
  $|v|_{1,h,D} =\big[\sum_{K\in \Tau_h(D)} \|\grad v\|_{0,K}^2\big]^{1/2}.$
The nonconforming Galerkin method for \eqref{eq:u0} reads as: find $u_h^0\in
V_{h,0}^{P_1}(\O)$ such that
\begin{align}
\sum_{K\in\Tau_h(\O)} (\mbA^0 \grad u_h^0, \grad v_h) = (f,v_h),\quad
v_h\in V_{h,0}^{P_1}(\O).
\end{align}
The standard error analysis for nonconforming elements implies
the following a priori error estimate, see \cite{park-sheen-p1quad, dssy-nc-ell},
\begin{align}
|u^0_h-u^0|_{1,h,\O} \le Ch |u^0|_2.
\end{align}

\section{FEHMM based on nonconforming spaces}\label{sec:fehmm-nc}
In this section we introduce a FEHMM scheme based on nonconforming finite spaces for the multiscale elliptic problem \eqref{eq:multiscale-elliptic}. 
We follow the framework of FEHMM \cite{abdulle2009finite, abdulle2012discontinuous} with slight modification for nonconforming function spaces.
Let $\Tau_H := \Tau_H(\O)$ be a regular triangulation of $\O$ with
quasi-uniform quadrilaterals ($d=2$) or hexahedrons ($d=3$).
Define the macro mesh parameter $H:=\max_{K\in \Tau_H} \diam(K)$.
For each macro element $K_H \in \Tau_H$, let $\mcF(K_H)$ denote the set of its faces. 
The set of all faces, of all interior faces and of all boundary faces are denoted by $\mcF_H$, $\mcF_H^i$ and $\mcF_H^b$, respectively.
Let $F_{K_H}: \hat{K} \rightarrow K_H $ be a bilinear transformation
from the reference domain onto $K_H$. Notice that for the $P_1$--NC
quadrilateral element, one may use the nonparametric scheme \cite{park-sheen-p1quad}.
Set 
\begin{eqnarray}
V= H^1_0(\O)\quad\text{and}\quad V_H = V^{P_1}_{H,0}(\O)
\end{eqnarray}
and denote the macro mesh-dependent (semi-)norm on $V+V_H$ by
$
  \brknmH{\cdot} := \big[\sum_{K_H \in \Tau_H} | \cdot |_{1,K_H}^2 \big]^{1/2}.
$

To formulate the FEHMM scheme,
we need a quadrature formula which consists of $I$ points and weights $(\mbx_i, \omega_i )_{i=1}^I$ 
on each element $K_H \in \Tau_H$ such that
\begin{align} \label{eq:quad}
\sum_{i=1}^I \omega_i \nabla v(\mbx_i) \cdot \nabla w(\mbx_i) 
=
	\int_{K_H} \nabla v \cdot \nabla w \dx
	\quad \forany v, w \in \mathcal{P}_1(K_H).
\end{align}

\begin{remark}
The above characteristics of the quadrature formula are useful
to prove the existence and uniqueness of the solution
as well as optimal error estimates in \secref{sec:analysis}.
\end{remark}

On each element $K_H \in \Tau_H$ we define $I$ sampling domains $K_{\delta,i} := \mbx_i + [-\delta/2, \delta/2]^d $ corresponding to each quadrature point $\mbx_i$ for given $\delta \ll 1$. 
The size of the sampling domains $\delta$ should be chosen to be comparable with $\vep$.
The most trivial case is $\delta=\vep$, but not always. The effect of various $\delta$ will be mentioned in \secref{sec:modeling-error}.
On each sampling domain we consider a micro triangulation to deal with a bundle of micro problems on it.
Let $\Tau_h(K_{\delta,i})$ be a uniform triangulation of a sampling domain $K_{\delta,i}$ consisting of quadrilateral elements and $h:=\max_{K\in \Tau_h(K_{\delta,i})} \diam(K)$ the micro mesh parameter.
Each micro element $K_h \in \Tau_h(K_{\delta,i})$ has a bilinear transformation $F_{K_h}: \hat{K} \rightarrow K_h$ such that $F_{K_h} (\hat{K}) = K_h$.
$\mcF(K_h)$ denotes the set of all faces of $K_h$.

On each sampling domain $K_{\delta,i}$ we will consider two
micro function spaces, namely, a continuous function space
$\tH(K_{\delta,i})$ and a discrete space $\tH_h(K_{\delta,i})$ which are determined by a choice of macro-micro coupling condition we use. 
If the coefficient tensor $\mbA^\vep$ in \eqref{eq:multiscale-elliptic} has a periodic property,
then we can impose a periodic coupling condition. 
On the other hand, a Dirichlet coupling condition can be used for general cases.
Respective to the choice we define two micro function spaces by
\begin{subequations}
\begin{align}\label{def:W}
\tH(K_{\delta,i}) &= \begin{cases} \tH^1_{\#}(K_{\delta,i}), 
&  \quad\text{periodic case}, \\
	H^1_0(K_{\delta,i}), & \quad\text{Dirichlet BC case}, \\
\end{cases}\\
\label{def:Wh}
\tH_h(K_{\delta,i}) &= 
	\begin{cases}
	V^{P_1}_{h,\#}(K_{\delta,i})& \quad\text{periodic case}, \\
	V^{P_1}_{h,0}(K_{\delta,i})&  \quad\text{Dirichlet BC case}.
	\end{cases}
\end{align}
\end{subequations}
The characteristics of the discrete space for periodic BC coupling cases are described in \cite{yim-sheen-p1nc-per}.
In both periodic and Dirichlet BC coupling cases, the micro mesh-dependent (semi-)norm on $\tH(K_{\delta,i})+\tH_h(K_{\delta,i})$ is
defined by
$
\brknmh{\cdot} := \big[ \sum_{K_h \in \Tau_h(K_{\delta,i})} | \cdot |_{1,K_h}^2 \big]^{1/2}.
$
The expression $K_{\delta,i}$ in notations will be omitted if there is no
ambiguity of choice for sampling domains.

For the sake of convenience, introduce the two bilinear forms, 
$a^{K_{\delta,i}}: H^1(K_{\delta,i})\times H^1(K_{\delta,i}) \to \mbbR $ and
$a^{K_{\delta,i}}_h: V^{P_1}_h(K_{\delta,i})\times
  V^{P_1}_h(K_{\delta,i}) \to \mbbR $ 
by
\begin{align*}
  a^{K_{\delta,i}}(u, v) & = \int_{K_{\delta,i}} \mbA^\vep
  \grad u \cdot \grad v \dx 
\forany u, v \in H^1(K_{\delta,i}),
\\
  a^{K_{\delta,i}}_h(u_h, v_h) & = \sum_{K_h \in
    \Tau_h(K_{\delta,i})} \int_{K_h} \mbA^\vep \grad u_h \cdot
  \grad v_h \dx
\forany u_h, v_h \in V^{P_1}_h(K_{\delta,i}).
\end{align*}
Also define two bilinear forms  $\oa_{H}$ and 
 $a_{H}: V_H\times V_H \rightarrow \mbbR$ as follows:
for all $u_H,v_H \in V_H,$
\begin{subeqnarray}
\slabel{eq:weakform-conti-bilinear}
 \oa_{H}(u_H, v_H)  &=& 
 	\sum_{K_H \in \Tau_H} \sum_{i=1}^I \frac{\omega_i}{\delta^d} 
 	\int_{K_{\delta,i}} \mbA^\vep \grad u^m \cdot \grad v^m
        \dx, \\
a_{H}(u_H, v_H) &=& \sum_{K_H \in \Tau_H} \sum_{i=1}^I \frac{\omega_i}{\delta^d} \sum_{K_h \in \Tau_h(K_{\delta,i})} \int_{K_h} \mbA^\vep \grad u_h^m \cdot \grad v_h^m \dx,
\end{subeqnarray}
where $u^m, v^m, u_h^m,$ and $v_h^m$
are the solutions of the continuous and discrete micro problems with constraints $u_H$ and $v_H$, respectively,
on each sampling domain $K_{\delta,i}$ in $K_H \in \Tau_H$
defined as follows:
for given $w_H\in V_H,$ $w^m \in w_H + \tH(K_{\delta,i})$ and
$w_h^m \in w_H + \tH_h(K_{\delta,i})$ are the solutions of
\begin{subeqnarray} \label{eq:micro-problem}
a^{K_{\delta,i}}(w^m, z)
&=& 0 \forany z \in \tH(K_{\delta,i}),\slabel{eq:micro-problem-a}\\
a^{K_{\delta,i}}_h(w_h^m, z_h)
&=& 0 \forany z_h \in \tH_h(K_{\delta,i}). \slabel{eq:micro-problem-b}
\end{subeqnarray}
%
In the above expressions $w_H + \tH(K_{\delta,i})$ and $w_H + \tH_h(K_{\delta,i})$,
$w_H$ actually means $\left.w_H\right|_{K_{\delta,i}}$, the function restricted to $K_{\delta,i}$.
However, here and in what follows,
we use this abusive notation for the sake of simple expressions if context determines proper range of given function.

\begin{remark}\label{rem:no-linearization}
By following a typical FEHMM framework,
one needs to consider $w^{lin}_H$, a linearization of $w_H$ at $\mbx_i$,
instead of $w_H$ itself in order to get $w^m$ and $w^m_h$ in \eqref{eq:micro-problem}.
In our discussion, however, such a linearization is unnecessary
because the finite element space which we are considering consists of
piecewise linear functions. This is one of the major advantages of
using the $P_1$--NC quadrilateral element.
\end{remark}

A nonconforming FEHMM weak formulation of the problem \eqref{eq:multiscale-elliptic}
is now ready to be stated as follows:
\paragraph*{\bf (Main Weak Formulation)}
{\it Find $u_H \in V_H$ such that}
\begin{align}\label{eq:uH}
a_{H} (u_H, w_H) = (f,w_H) \ \forany w_H \in V_H.
\end{align}

\begin{figure}[!t]
\centering
\epsfig{figure=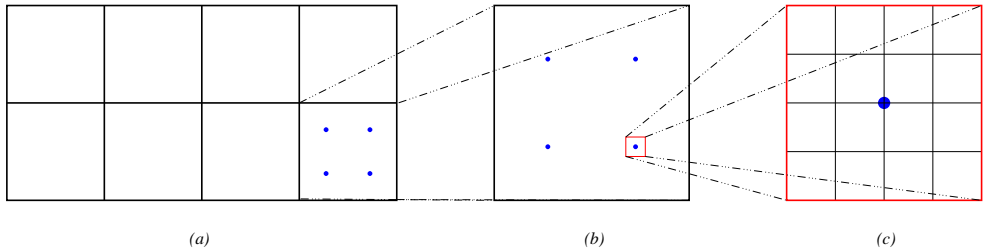,width=1.0\textwidth}
\caption{The hierarchy of geometric objects in FEHMM scheme. (a) Domain $\O$ and its triangulation $\Tau_H$ (b) Macro element $K_H$ (c) Sampling domain $K_{\delta,i}$ surrounding a quadrature point $\mbx_i$ and its triangulation $\Tau_h$ consisting of micro elements $K_h$}
\label{fig:geom-hierarchy}
\end{figure}


Several micro functions will be introduced, which will be used in \secref{sec:analysis}.
For $j=1,\cdots,d,$
let
$\psi^j=\psi^j(\mbx) \in \tH(K_{\delta,i})$ and
$\psi^j_h = \psi^j_h(\mbx) \in \tH_h(K_{\delta,i})$ be the solutions of 
the following micro problems:
\begin{subeqnarray}\slabel{eq:psi}
  a^{K_{\delta,i}}(\psi^j, z) &=& - \int_{K_{\delta,i}} \mbA^\vep \mbe_j \cdot \grad z \dx \forany z \in \tH(K_{\delta,i}),\\
\slabel{eq:psi_h}
  a^{K_{\delta,i}}_h(\psi^j_h, z_h) &=& - \sum_{K_h \in \Tau_h(K_{\delta,i})} \int_{K_h} \mbA^\vep \mbe_j \cdot \grad z_h
  \dx
  \forany z_h \in \tH_h(K_{\delta,i}).
\end{subeqnarray}
Later, $\psi^j_h,j=1,\cdots,d,$ play as basis functions for the
solution space of the micro problem \eqref{eq:micro-problem}
on each sampling domain $K_{\delta,i}.$
Instead of $\psi^j$ and $\psi^j_h$, 
sometimes the following function notations are useful:
\begin{eqnarray}\label{eq:phi}
\varphi^j(\mbx):=\psi^j(\mbx) + x_j\quad\text{and }  
\varphi^j_h(\mbx):=\psi^j_h(\mbx) + x_j.
\end{eqnarray}

Notice that the above \eqref{eq:psi} and \eqref{eq:psi_h}
are equivalent to
\begin{subeqnarray}\label{eq:phi-phi_h}
  a^{K_{\delta,i}}(\varphi^j, z) &=& 0
\forany z\in  \tH(K_{\delta,i}),\\
  a^{K_{\delta,i}}_h(\varphi^j_h, z_h) &=& 0
  \forany z_h \in \tH_h(K_{\delta,i}).
\end{subeqnarray}

\begin{remark}
Indeed, $\psi^j$ and $\psi^j_h$ are nothing but
$\psi^j = (x_j)^m - x_j$ and $\psi^j_h = (x_j)^m_h - x_j$, respectively,
with the meaning of the superscript `$m$' as same as in \eqref{eq:micro-problem}.
Moreover, $\varphi^j = (x_j)^m$ and $\varphi^j_h = (x_j)^m_h$.
\end{remark}

We also introduce several weak formulations which are used for analysis.
A weak formulation of the homogenized problem \eqref{eq:u0} is given
as to find $u^0 \in V$ such that
\[
a^0(u^0,v) = (f,v) \forany v \in V,
\]
where
\begin{align}
 a^0(v,w)  = \int_{\O} \mbA^0(\mbx) \grad v \cdot \grad w \dx\quad
\forany v,w \in V.
\end{align}
Also, a weak formulation of the homogenized problem \eqref{eq:u0} with quadrature rule in macro scale, corresponding
to \eqref{eq:uH}, can be defined
as to find $u^0_H \in V_H$ fulfilling
\begin{eqnarray}\label{eq:u^0_H}
a_{H}^0(u^0_H,v_H) = (f,v_H) \forany v_H \in V_H,
\end{eqnarray}
where
\begin{align}\label{eq:a^0_H}
 a_{H}^0(v_H,w_H) 
= 
 	\sum_{K_H \in \Tau_H} \sum_{i=1}^I \omega_i
        \ \mbA^0(\mbx_i) \grad v_H(\mbx_i) \cdot \grad w_H(\mbx_i)\quad\forany
v_H,w_H \in V_H.
\end{align}
Finally, recalling the bilinear form $\oa_{H}:V_H\times V_H
\rightarrow \mbbR$ defined in \eqref{eq:weakform-conti-bilinear}, define
a {\it semi-discrete FEHMM solution} $\ou_H \in V_H$ which satisfy
\begin{align}\label{eq:overlineuH}
 \oa_{H}(\ou_H,v_H) = (f,v_H) \forany v_H \in V_H.
\end{align}



\section{Fundamental properties of nonconforming HMM}\label{sec:analysis}

For our analysis, we mainly follow the framework of \cite{abdulle2012discontinuous}.
We want to emphasize the changing parts only due to the use of the $P_1$--nonconforming quadrilateral finite element space.

\subsection{Existence and uniqueness}
We begin with the proof of the existence and uniqueness of solutions
to 
\eqref{eq:uH} and \eqref{eq:overlineuH}.
\begin{lemma}
Let $v^m_h$ be the solution of the micro problem \eqref{eq:micro-problem-b}
with constraint $v_H$ on a sampling domain $K_{\delta,i}$.
Then
\begin{align*}
|v_H|_{1,K_{\delta,i}} \le \brknmh{v^m_h} \le
\frac{\Lambda}{\lambda}|v_H|_{1,K_{\delta,i}}.
\end{align*}
\end{lemma}

\begin{proof}
Utilizing the fact that $v_H$ is linear on $K$
for all $K\in \Tau_H$ 
and \eqref{def:Wh}, we have
\begin{align*}
0
&\le \sum_{K_h \in \Tau_h} 
\int_{K_h} \grad (v^m_h-v_H) \cdot \grad (v^m_h-v_H) \dx \\
&= 
	\sum_{K_h \in \Tau_h} 
\int_{K_h} |\grad v^m_h|^2 
	- |\grad v_H|^2 
	- 2 \grad (v^m_h-v_H) \cdot \grad v_H \dx \\
&= 	
\Big[\sum_{K_h \in \Tau_h} 
	\int_{K_h} |\grad v^m_h|^2 \dx 
	- \int_{K_{\delta,i}} |\grad v_H|^2 \dx \Big]
	- 2 \grad v_H \cdot \sum_{K_h \in \Tau_h} \int_{\p K_h} \bnu_{K_h} (v^m_h-v_H) \dsig
\end{align*}
where $\bnu_{K_h}$ denotes the unit outward normal to $K_h$.
Since $v^m_h-v_H \in \tH_h(K_{\delta,i})$, the last term in the above equation
vanishes. 
Consequently, we get
\begin{align*}
\int_{K_{\delta,i}} |\grad v_H|^2 \dx 
\le 
	\sum_{K_h \in \Tau_h} \int_{K_h} |\grad v^m_h|^2 \dx.
\end{align*}

On the other hand, due to the ellipticity of $\mbA^\vep$, we have
\begin{align*}
0
&\le 
	\sum_{K_h \in \Tau_h} \int_{K_h} \mbA^\vep \grad (v^m_h-v_H) \cdot \grad (v^m_h-v_H) \dx \\
&= 
	\sum_{K_h \in \Tau_h} \int_{K_h} \mbA^\vep \grad v_H \cdot \grad v_H 
	- \mbA^\vep \grad v^m_h \cdot \grad v^m_h \\
&\qquad	+ \mbA^\vep \grad (v^m_h-v_H) \cdot \grad v^m_h 
	+ \mbA^\vep \grad v^m_h \cdot \grad (v^m_h-v_H) \dx.
\end{align*}
Due to the definition of $v^m_h$ in \eqref{eq:micro-problem} and symmetry of $\mbA^\vep$, the last two terms vanish. Thus we get
\begin{align*}
\sum_{K_h \in \Tau_h} \int_{K_h} \mbA^\vep \grad v^m_h \cdot \grad v^m_h \dx 
\le 
	\int_{K_{\delta,i}} \mbA^\vep \grad v_H \cdot \grad v_H \dx.
\end{align*}
The uniform ellipticity and boundedness of $\mbA^\vep$ imply the desired inequality.
\end{proof}
Thanks to the properties of the quadrature formula \eqref{eq:quad}, the bilinear form
$a_{H}$ is bounded and coercive in $V_H$. Therefore the
existence and uniqueness of the solution $u_H$ to \eqref{eq:uH} is
guaranteed by the Lax-Milgram Lemma. Thus, we have

\begin{theorem}
There exists a unique solution $u_H$ to the problem \eqref{eq:uH}.
\end{theorem}

Similarly, the coercivity and boundedness of the
bilinear form $\oa_H$ can be obtained immediately, and thus one also obtains the
existence and uniqueness of the solution $\ou_H$ to
\eqref{eq:overlineuH}, stated as follows:
\begin{theorem}
There exists a unique solution $\ou_H\in V_H$ to Problem \eqref{eq:overlineuH}.
\end{theorem}

\subsection{Recovered homogenized tensors}
The recovered homogenized tensors 
$\overline{\mbA}^0_{K_{\delta,i}}$ and
$\mbA^0_{K_{\delta,i}}$ on a sampling domain $K_{\delta,i}$ are defined by
\begin{subequations}\label{eq:recovered-homogenized-tensor}
\begin{align}
\overline{\mbA}^0_{K_{\delta,i}} 
&= 
	\delta^{-d} \int_{K_{\delta,i}} \mbA^\vep(\mbx) \left(I + \mbJ_{\psi}^T \right) \dx,\\
  \mbA^0_{K_{\delta,i}} 
&= 
	\delta^{-d} \sum_{K_h \in \Tau_h(K_{\delta,i})} \int_{K_h} \mbA^\vep(\mbx) \left(I + \mbJ_{\psi_h}^T \right) \dx,
\end{align}
\end{subequations}
where $\mbJ_{\psi}$ and $\mbJ_{\psi_h}$ are $d \times d$ matrices
defined by
$\left(J_{\psi}\right)_{jk} 
= 	\frac{\p \psi^{j}}{\p x_k}
$
and
$
\left(J_{\psi_h}\right)_{jk} 
= 
	\frac{\p \psi^j_h}{\p x_k},\ 1\le j,k\le d,$
respectively. The following propositions show the essential characteristic of two recovered homogenized tensors.

\begin{proposition}\label{prop:rec-homojk} The following
  representations hold for
$\big(\overline{A}^0_{K_{\delta,i}}\big)_{jk}$ and
  $\big(A^0_{K_{\delta,i}}\big)_{jk}$:
\begin{align*}
\big(\overline{A}^0_{K_{\delta,i}}\big)_{jk}
&= 
	\delta^{-d} \int_{K_{\delta,i}}
	\sum_{\ell=1}^d A^\vep_{j\ell} \frac{\p \varphi^k}{\p x_\ell} \dx
= 
	\delta^{-d} \int_{K_{\delta,i}}
	\mbA^\vep \grad \varphi^k \cdot \grad x_j \dx \\
&= 
	\delta^{-d} \int_{K_{\delta,i}}
	\mbA^\vep \grad \varphi^k \cdot \grad \varphi^j \dx,\\
\big(A^0_{K_{\delta,i}}\big)_{jk}
&= 
	\delta^{-d} \sum_{K_h \in \Tau_h} \int_{K_h}
	\sum_{\ell=1}^d A^\vep_{j\ell} \frac{\p \varphi^k_h}{\p x_\ell} \dx
= 
	\delta^{-d} \sum_{K_h \in \Tau_h} \int_{K_h}
	\mbA^\vep \grad \varphi^k_h \cdot \grad x_j \dx \\
&= 
	\delta^{-d} \sum_{K_h \in \Tau_h} \int_{K_h} 
	\mbA^\vep \grad \varphi^k_h \cdot \grad \varphi^j_h \dx.
\end{align*}
\end{proposition}
\begin{proof}
The above identities are obtained by the application of the definitions
\eqref{eq:recovered-homogenized-tensor} of the recovered homogenized
tensors, those
\eqref{eq:phi-phi_h} of $\varphi^j$ and $\varphi_h^j$, and those of $J_{\psi}$ and
$J_{\psi_h},$ and Equations \eqref{eq:micro-problem}.
\end{proof}

\begin{proposition}\label{prop:recovered-homo}
Let $u^m$ and $v^m$ be the solutions of the continuous micro problem
\eqref{eq:micro-problem-a} corresponding to the macro constraints $u_H$ and $v_H$, respectively, on $K_{\delta,i}$.
Then the following holds:
\begin{equation}\label{eq:uHvH-2}
\delta^{-d} \int_{K_{\delta,i}} \mbA^\vep \grad u^m \cdot \grad v^m \dx
= \overline{\mbA}^0_{K_{\delta,i}} \grad u_H \cdot \grad v_H.
\end{equation}
Similarly, let $u^m_h$ and $v^m_h$ be the solutions of the discrete micro problem 
\eqref{eq:micro-problem-b} corresponding to the macro constraints
$u_H$ and $v_H$, respectively, on $K_{\delta,i}$. Then the following holds:
\begin{equation}\label{eq:uHvH}
\delta^{-d} \sum_{K_h \in \Tau_h(K_{\delta,i})} \int_{K_h} \mbA^\vep \grad u^m_h \cdot \grad v^m_h \dx
= 
	\mbA^0_{K_{\delta,i}} \grad u_H \cdot \grad v_H.
\end{equation}
\end{proposition}

\begin{proof}
We will show \eqref{eq:uHvH} only, as the proof of \eqref{eq:uHvH-2}
follows essentially by the same arguments.
Since $u^m_h$ is the solution of \eqref{eq:micro-problem-b},
and $v^m_h-v_H \in \tH_h(K_{\delta,i})$, it holds
\begin{align}\label{eq:macro-1}
 \delta^{-d} \sum_{K_h \in \Tau_h} \int_{K_h} \mbA^\vep \grad u^m_h \cdot \grad v^m_h \dx
=  \delta^{-d} \sum_{K_h \in \Tau_h} \int_{K_h} \mbA^\vep \grad u^m_h \cdot \grad v_H \dx.
\end{align}
Since $\grad u_H$ is constant, $u^m_h$
is represented by a linear combination of the basis functions
$\psi^j_h$ as
$
  u^m_h = u_H + \sum_{j=1}^d \psi^j_h \ \frac{\p u_H}{\p x_j}.
$
By plugging this representation into \eqref{eq:macro-1}, we have
\begin{align*}
 \delta^{-d} & \sum_{K_h \in \Tau_h} \int_{K_h} \mbA^\vep 
 \Big( \grad u_H + \sum_{j=1}^d \grad \psi^j_h \ \frac{\p u_H}{\p x_j} \Big) 
 \cdot \grad v_H  \dx \\
=& 
 	\delta^{-d} \sum_{K_h \in \Tau_h} \int_{K_h} \mbA^\vep 
 	\left( I + J_{\psi_h}^T \right) \grad u_H \cdot \grad v_H \dx
=
	\mbA^0_{K_{\delta,i}} \grad u_H \cdot \grad v_H.
\end{align*}
This completes the proof.
\end{proof}

\begin{remark}
\propref{prop:recovered-homo} implies that $\mbA^0_{K_{\delta,i}}$
indeed plays a role as the homogenized tensor on each sampling domain $K_{\delta,i}$ numerically.
\end{remark}

\subsection{The case of periodic coupling}
In this section, main ingredients of error analysis are
provided under periodic assumptions.
Recall the definitions of $\psi^j$ and $\varphi^j$ in \eqref{eq:psi}
and \eqref{eq:phi}.
 The following two assumptions will be
taken into consideration in this section.
\begin{assumption}[Periodic coupling]\label{ass:periodic}
\
\begin{enumerate}
\item $\mbA^\vep(\mbx) := \mbA(\mbx,\frac{\mbx}{\vep})$ where $\mbA(\mbx,\cdot)$ is $Y$-periodic with $Y =
  [0,1]^d$ and $\mbA(\mbx,\cdot) \in \mbW^{1,\infty}(Y)$.
\item On each sampling domain $K_{\delta,i}$, the solution of the
  micro problem
  \eqref{eq:psi} with periodic coupling \eqref{def:W} has regularity
  $\psi^j \in H^2(K_{\delta,i})$ and
  $\mbA^\vep \grad \varphi^j \in [H^1(K_{\delta,i})]^d$.
\end{enumerate}
\end{assumption}

First, we note the following lemmas. The proofs follow the main framework and are omitted.
\begin{lemma}\label{lemma:divflux} Under \assref{ass:periodic},
$\div \left( \mbA^\vep \grad \varphi^j \right) =0 $ on $K_{\delta,i}$
  {\it a.e.}
\end{lemma}

\begin{lemma}\label{lem:micro-regularity}
Under \assref{ass:periodic}, it holds
\begin{align}
|\mbA^\vep \grad \varphi^j|_{1,K_{\delta,i}} 
\le C 
	\delta^{\frac{d}{2}} \vep^{-1}.
\end{align}
\end{lemma}

The following proposition is a discretization error estimation of the micro basis function $\psi^j$.

\begin{proposition}\label{prop:basis-micro-error}
Under \assref{ass:periodic}, it holds
\begin{align}\label{eq:basis-micro-error}
\brknmh{\psi^j - \psi^j_h}
& \le C 
	h \delta^{\frac{d}{2}} \vep^{-1}.
\end{align}  
\end{proposition}
\begin{proof}
The Second Strang Lemma \cite{strang-fix, braess} for the micro problems \eqref{eq:psi_h} and \eqref{eq:psi} implies that
\begin{align*}
\brknmh{\psi^j - \psi^j_h}
& \le C 
	\Big( \inf_{v_h \in \tH_h(K_{\delta,i})} \brknmh{\psi^j - v_h}
  \\
  &\qquad
	+ \sup_{w_h \in \tH_h(K_{\delta,i})} \frac{\big|a^{K_{\delta,i}}_h(\psi^j, w_h) - a^{K_{\delta,i}}_h(\psi^j_h, w_h)\big|}{\brknmh{w_h}} \Big).
\end{align*}
The first term represents the best approximation error of $\psi^j,$ 
which is of $\Oh(h)$ due to the standard approximation property of nonconforming element spaces. 
The second term, so-called the consistency error, measures nonconformity of the finite element space.
Denoting  by $L(w_h)$ the numerator of the second term, from the
definitions \eqref{eq:psi_h} and \eqref{eq:phi} of $\psi^j_h$ and $\varphi^j$ one sees that
\begin{align*}
\left| L(w_h) \right| 
& = 
	\Big| \sum_{K_h \in \Tau_h} \int_{K_h} \mbA^\vep \grad \psi^j \cdot \grad w_h \dx 
  +\sum_{K_h \in \Tau_h} \int_{K_h} \mbA^\vep \mbe_j \cdot \grad w_h \dx \Big| \\
& = 
	\Big| \sum_{K_h \in \Tau_h} \int_{K_h} \mbA^\vep \grad \varphi^j \cdot \grad w_h \dx \Big| \\
& = 
	\Big| \sum_{K_h \in \Tau_h} \int_{\p K_h} \bnu_{K_h} \cdot \left(\mbA^\vep \grad \varphi^j \right) w_h \dsig
	- \sum_{K_h \in \Tau_h} \int_{K_h} \div \left( \mbA^\vep
  \grad \varphi^j \right) w_h \dx \Big|.
  \end{align*}
Due to \lemref{lemma:divflux}, the above equation is bounded by only the first term such that
\begin{align*}
	\Big| \sum_{K_h \in \Tau_h} \int_{\p K_h} \bnu_{K_h}
	\cdot \mbA^\vep \grad \varphi^j w_h \dsig \Big|
& = 
	\Big| \sum_{K_h \in \Tau_h} \sum_{\tau\in \mcF(K_h)}\int_{\tau} \bnu_{K_h} \cdot
	\mbA^\vep \grad \varphi^j w_h \dsig \Big| \\
& \le
	\sum_{\tau\in\mcF_h} \Big| \int_\tau \bnu_{\tau} \cdot
	\mbA^\vep \grad \varphi^j [w_h]_\tau \dsig \Big|.
\end{align*}
Thanks to the regularity in \assref{ass:periodic}, the Bramble--Hilbert lemma,
and the trace theorem, one sees that
\begin{align*}
\int_{\tau} \bnu_\tau \cdot \mbA^\vep \grad \varphi^j [w_h]_\tau \dsig
& = 
	\int_{\tau} \bnu_\tau \cdot \mbA^\vep \grad \varphi^j \big[ w_h - \avint_{\tau} w_h \big]_\tau \dsig \\
& = 
	\int_{\tau} \bnu_\tau \cdot \big[ \mbA^\vep \grad \varphi^j - \avint_{\tau} \big(\mbA^\vep \grad \varphi^j \big) \big]
	\big[ w_h - \avint_{\tau} w_h \big]_\tau \dsig \\
& \le 
	\big[ \int_{\tau} \big| \mbA^\vep \grad \varphi^j - \avint_{\tau} \left(\mbA^\vep \grad \varphi^j \right) \big|^2 \dsig \big]^{\frac{1}{2}}
	\big[ \int_{\tau} \big| w_h -  \avint_{\tau} w_h
  \big|^2 \dsig \big]^{\frac{1}{2}}\\
& \le Ch^{\frac12} \left| \mbA^\vep\grad \varphi^j \right|_{1/2,\tau}
h^{\frac12}  \left| w_h \right|_{1/2,\tau}\\
  &\le Ch \left| \mbA^\vep \grad \varphi^j \right|_{1,K_h}
    	 \left| w_h \right|_{1,K_h},\quad \tau\in \mcF(K_h).
\end{align*}
Consequently, we have
\begin{align*}
\left| L(w_h) \right|
& \le C
	\sum_{K_h \in \Tau_h} h |\mbA^\vep \grad \varphi^j|_{1,K_h} 
	| w_h |_{1,K_h}\\
& \le C
	h \big[ \sum_{K_h \in \Tau_h} |\mbA^\vep \grad \varphi^j|_{1,K_h}^2 \big]^{1/2} 
	\big[ \sum_{K_h \in \Tau_h} | w_h |^2_{1,K_h} \big]^{1/2} \\
& = C
	h |\mbA^\vep \grad \varphi^j|_{1,K_{\delta,i}} \brknmh{w_h}.
\end{align*}
Finally, \lemref{lem:micro-regularity} and the regularity result imply the desired error estimate:
\begin{align*}
\brknmh{\psi^j - \psi^j_h}
& \le C 
	h \left( |\psi^j|_{2,K_{\delta,i}} + |\mbA^\vep \grad \varphi^j|_{1,h,K_{\delta,i}} \right)
 \le C 
	h \delta^{\frac{d}{2}} \vep^{-1}.
\end{align*}
\end{proof}

The following proposition shows difference between the recovered homogenized tensors.
\begin{proposition}\label{prop:recovered-tensor-micro-error}
Under \assref{ass:periodic}, the following holds:
\begin{align}\label{eq:recovered-tensor-micro-error}
\sup_{K_{\delta,i}} \| \overline{\mbA}^0_{K_{\delta,i}}-\mbA^0_{K_{\delta,i}} \|_2
\le C 
	\Big( \frac{h}{\vep} \Big)^2,
\end{align}
where $\|\cdot\|_2$ denotes the matrix 2--norm.
\end{proposition}

\begin{proof}
Invoking  \propref{prop:rec-homojk}, we have
\begin{align*}
&\big|\big(\mbA^0_{K_{\delta,i}}\big)_{jk}-\big(\overline{\mbA}^0_{K_{\delta,i}}\big)_{jk} \big| \\
& = 
	\Big| \delta^{-d} \sum_{K_h \in \Tau_h} \int_{K_h}
	\mbA^\vep \grad \varphi^k_h \cdot \grad \varphi^j_h 
	- \mbA^\vep \grad \varphi^{k} \cdot \grad \varphi^{j} \dx \Big| \\
& = 
	\Big| \delta^{-d} \sum_{K_h \in \Tau_h} \int_{K_h}
	\mbA^\vep \grad \big( \varphi^k_h - \varphi^k \big) 
	\cdot \grad \big( \varphi^j_h - \varphi^j \big) \nonumber \\
	& \qquad \qquad \qquad \qquad \qquad + \mbA^\vep \grad \varphi^k \cdot \grad \big( \varphi^j_h - \varphi^j \big) 
	+ \mbA^\vep \grad \big( \varphi^k_h - \varphi^k \big) \cdot \grad \varphi^j \dx \Big| \\
& \le 
	\delta^{-d} \Big| \sum_{K_h \in \Tau_h} \int_{K_h} 
	\mbA^\vep \grad \big( \varphi^k_h - \varphi^k \big) \cdot
  \grad \big( \varphi^j_h - \varphi^j \big) \dx \Big| \nonumber \\
	& \quad + \delta^{-d} \Big| \sum_{K_h \in \Tau_h} -
\int_{K_h} 
	\div \big( \mbA^\vep \grad \varphi^k \big) \big( \varphi^j_h - \varphi^j \big) \dx 
	+ \int_{\p K_h} \bnu_{K_h} \cdot \mbA^\vep \grad \varphi^k \big( \varphi^j_h - \varphi^j \big) \dsig \Big| \nonumber \\
	& \quad + \delta^{-d} \Big| \sum_{K_h \in \Tau_h} -
\int_{K_h} \div \big( \mbA^\vep \grad \varphi^j \big) \big( \varphi^k_h - \varphi^k \big) \dx 
	+ \int_{\p K_h} \bnu_{K_h} \cdot \mbA^\vep \grad \varphi^j \big(\varphi^k_h - \varphi^k \big) \dsig \Big|.
\end{align*}
The last inequality is due to an integration by parts and the symmetry
of $\mbA^\vep$. By \lemrefs{lemma:divflux}{lem:micro-regularity},
and \propref{prop:basis-micro-error} with proper modifications, the
terms in the last inequality is further bounded by
\begin{align*}
& \le 
	\frac{\Lambda}{\delta^d} \brknmh{\varphi^k_h - \varphi^k} \brknmh{\varphi^j_h - \varphi^j} \nonumber \\
	& \qquad + \delta^{-d} \Big| \sum_{K_h \in \Tau_h} 
	\int_{\p K_h} \bnu_{K_h} \cdot \mbA^\vep \grad \varphi^k \big(\varphi^j_h - \varphi^j \big) \dsig \Big| \nonumber \\
	& \qquad + \delta^{-d} \Big| \sum_{K_h \in \Tau_h}
	\int_{\p K_h} \bnu_{K_h} \cdot \mbA^\vep \grad \varphi^j \big(\varphi^k_h - \varphi^k \big) \dsig \Big| \\
& \le 
	\frac{\Lambda}{\delta^d} \brknmh{\varphi^k_h - \varphi^k} \brknmh{\varphi^j_h - \varphi^j} \nonumber \\
	& \qquad + \delta^{-d} h |\mbA^\vep \grad \varphi^k|_{1,K_{\delta,i}} \brknmh{\varphi^j_h - \varphi^j}
	+ \delta^{-d} h |\mbA^\vep \grad \varphi^j|_{1,K_{\delta,i}} \brknmh{\varphi^k_h - \varphi^k} \\
& \le C 
	\Big(\frac{h}{\vep}\Big)^2.
\end{align*}
This completes the proof.
\end{proof}

\subsection{The case of Dirichlet coupling}

In this subsection we consider the following assumptions
and the Dirichlet coupling condition for micro problems.

\begin{assumption}[Dirichlet coupling]\label{ass:dirichlet}
\
\begin{enumerate}
\item $\mbA^\vep(\mathbf{x}) \in \tH^{1,\infty}(K_H)$ with $|A^\vep_{jk}|_{0,\infty,K_H} \le C$ and $|\grad
A^\vep_{jk}|_{0,\infty,K_H} \le C / \vep$ for all $K_H \in \Tau_H$.
\item On each sampling domain $K_{\delta,i}$, the solution of the micro problem \eqref{eq:psi} with Dirichlet coupling \eqref{def:W} has regularity $\psi^j \in H^2(K_{\delta,i})$ and $\mbA^\vep \grad \varphi^j \in [H^1(K_{\delta,i})]^2$.
\end{enumerate}
\end{assumption}

The same arguments as in the case of periodic coupling in Section 4.3
leads to the corresponding results to
\proprefs{prop:basis-micro-error}{prop:recovered-tensor-micro-error}
hold under \assref{ass:dirichlet}. Omitting the details of proof, we
state the results as follows:

\begin{proposition}\label{prop:basis-micro-error-diri}
Under \assref{ass:dirichlet}, it holds
\begin{align}\label{eq:basis-micro-error-diri}
\brknmh{\psi^j - \psi^j_h}
& \le C 
	h \delta^{\frac{d}{2}} \vep^{-1}.
\end{align}  
\end{proposition}

\begin{proposition}\label{prop:recovered-tensor-micro-error-diri}
Under \assref{ass:dirichlet}, the following holds:
\begin{align}\label{eq:recovered-tensor-micro-error-diri}
\sup_{K_{\delta,i}} \| \overline{\mbA}^0_{K_{\delta,i}}-\mbA^0_{K_{\delta,i}} \|_2
\le C 
	\Big( \frac{h}{\vep} \Big)^2.
\end{align}
\end{proposition}

\subsection{A priori error estimate}
The description below is following the analysis framework from \cite{abdulle2012discontinuous},
but adopting the finite element specific results above.
We do not repeat the details of proof in this section.

\subsubsection{The macro error}
Under sufficient regularity of $u^0$, for instance $H^2(\O)$, the standard
analysis for nonconforming finite elements \cite{park-sheen-p1quad} yields
\begin{align}\label{eq:macro-error}
\brknmH{u^{0}-u^0_H} \le C H \|u^{0}\|_2.
\end{align}

\subsubsection{The modeling error}\label{sec:modeling-error}
Since $\ou_H, u^0_H\in V_H$, using the uniform ellipticity of $\oa_H$
\eqref{eq:overlineuH}, and  \eqref{eq:u^0_H}, we have
\begin{align*}
\brknmH{u^0_H-\ou_H}^2 
&\le 
	\oa_{H}(u^0_H-\ou_H, u^0_H-\ou_H) \\
&= 
	\oa_{H}(u^0_H, u^0_H-\ou_H) -a_{H}^0(u^0_H, u^0_H-\ou_H).
\end{align*}
We deduce the following Strang-type inequality:
\begin{align}
\brknmH{u^0_H-\ou_H}
&\le 
	\sup_{w_H \in V_H} 
	\frac{\left| \oa_{H}(u^0_H,w_H)-a_{H}^0(u^0_H,w_H) \right|}{\brknmH{w_H}}.
\end{align}
Recalling \eqref{eq:weakform-conti-bilinear} and \eqref{eq:a^0_H}, and using
\propref{prop:recovered-homo}, we can bound the numerator as follows:
\begin{align*}
&\left| \oa_{H}(u^0_H,w_H)-a_{H}^0(u^0_H,w_H) \right| \\
&= 
	\Big| \sum_{K_H \in \Tau_H} \sum_{i=1}^I \omega_i \ \overline{\mbA}^0_{K_{\delta,i}} \grad {u^0_H} \cdot \grad {w_H}
	- \sum_{K_H \in \Tau_H} \sum_{i=1}^I \omega_i \ \mbA^0(\mbx_i) \grad u^0_H(\mbx_i) \cdot \grad w_H(\mbx_i) \Big| \\
&\le 
	\sum_{K_H \in \Tau_H} \sum_{i=1}^I \omega_i 
	\Big| \big(\overline{\mbA}^0_{K_{\delta,i}} - \mbA^0(\mbx_i) \big)
	\grad {u^0_H} \cdot \grad {w_H} \Big| \\
&\le 
	\sup_{K_{\delta,i}} \|\overline{\mbA}^0_{K_{\delta,i}}-\mbA^0(\mbx_i) \|_2 \ \brknmH{u^0_H} \ \brknmH{w_H}.
\end{align*}
Thus, we have
\begin{align}\label{eq:modeling-error}
\brknmH{u^0_H-\ou_H}
&\le C
	\sup_{K_{\delta,i}}
  \|\overline{\mbA}^0_{K_{\delta,i}}-\mbA^0(\mbx_i) \|_2\, \brknmH{u^0_H}.
\end{align}

As mentioned in \cite{abdulle2012discontinuous}, the following result is obtained in \cite{e2005analysisa}:
If we assume the local periodicity of $\mbA^\vep$, for instance under \assref{ass:periodic}, then 
\begin{align}
\sup_{K_{\delta,i}} \|\mbA^0-\overline{\mbA}^0_{K_{\delta,i}}\|_2
\le
\begin{cases}
C \vep &\text{if periodic coupling with } \delta/\vep \in \mathbb{N} \text{ is used for }\eqref{eq:micro-problem}, \\
C \left( \frac{\vep}{\delta} + \delta \right) &\text{if Dirichlet coupling or } \delta/\vep \not\in \mathbb{N} \text{ is used}.
\end{cases}
\end{align}

\subsubsection{The micro error}
Due to the uniform ellipticity of $a_H$, we have
\begin{align*}
\brknmH{\ou_H-u_H}^2 
&\le 
	a_{H}(\ou_H-u_H, \ou_H-u_H) \\
&= 
	a_{H}(\ou_H, \ou_H-u_H) - \oa_{H}(\ou_H, \ou_H-u_H).
\end{align*}
Therefore it holds
\begin{align}
\brknmH{\ou_H-u_H}
&\le 
	\sup_{w_H \in V_H} \frac{\left| a_{H}(\ou_H,w_H)-\oa_{H}(\ou_H,w_H) \right|}{\brknmH{w_H}}.
\end{align}
Again, recalling \eqref{eq:weakform-conti-bilinear} and
\eqref{eq:a^0_H}, we see that
an application of \proprefs{prop:recovered-homo}{prop:recovered-tensor-micro-error}
to the numerator gives the following bound:
\begin{align*}
&\left| a_{H}(\ou_H,w_H)-\oa_{H}(\ou_H,w_H) \right| \\
&= 
	\Big| \sum_{K_H \in \Tau_H} \sum_{i=1}^I \omega_i \
	\mbA^0_{K_{\delta,i}} \grad {\ou_H} \cdot \grad {w_H}
	- \sum_{K_H \in \Tau_H} \sum_{i=1}^I \omega_i \
	\overline{\mbA}^0_{K_{\delta,i}} \grad {\ou_H} \cdot \grad {w_H} \Big| \\
&\le 
	\sum_{K_H \in \Tau_H} \sum_{i=1}^I \omega_i  
	\Big| \big(\mbA^0_{K_{\delta,i}} - \overline{\mbA}^0_{K_{\delta,i}} \big)
	\grad {\ou_H} \cdot \grad {w_H} \Big| \\
&\le 
	\sup_{K_{\delta,i}} \|\mbA^0_{K_{\delta,i}}-\overline{\mbA}^0_{K_{\delta,i}} \|_2 \ \brknmH{\ou_H} \ \brknmH{w_H}
\le C
	\Big( \frac{h}{\vep} \Big)^2 \brknmH{\ou_H} \ \brknmH{w_H},
\end{align*}
and we have
\begin{align}\label{eq:micro-error}
\brknmH{\ou_H-u_H}
&\le C	\Big( \frac{h}{\vep} \Big)^2  \brknmH{\ou_H}.
\end{align}
	
\subsubsection{Main theorem for error estimates}
The usual Aubin-Nitsche duality argument provides $L^2$-norm error
estimates, see, for instance, \cite{e2005analysisa}. The derivation is
standard, omitted here,
but we list the estimates:
\begin{theorem}\label{thm:main}
Let $u^0$ and $u_H$ be the solutions of \eqref{eq:u0} and
\eqref{eq:uH}. Then, the followings hold under various assumptions: 
\begin{enumerate}
\item Under \assref{ass:periodic}, if periodic coupling with $\delta/\vep \in \mathbb{N}$ is used, then
\begin{subequations}\label{eq:main-periodic}
\begin{align}
\brknmH{u^0 - u_H}
\le C
	\Big( H + \vep + \big(\frac{h}{\vep}\big)^2 \Big), \\
\| u^0 - u_H \|_0
\le C
	\Big( H^2 + \vep + \big(\frac{h}{\vep}\big)^2 \Big);
\end{align}
\end{subequations}

\item Under \assref{ass:periodic}, 
if periodic coupling with $\delta/\vep \not \in \mathbb{N}$ is used, then
\begin{subequations}\label{eq:main-periodic-noninteger}
\begin{align}
\brknmH{u^0 - u_H}
\le C
	\Big( H + \big(\frac{\vep}{\delta} + \delta\big) + \big(\frac{h}{\vep}\big)^2 \Big), \\
\| u^0 - u_H \|_0
\le C
	\Big( H^2 + \big(\frac{\vep}{\delta} + \delta\big) +
        \big(\frac{h}{\vep}\big)^2 \Big);
\end{align}
\end{subequations}

\item Under \assref{ass:periodic} and \assref{ass:dirichlet},
if Dirichlet coupling is used, then
\begin{subequations}\label{eq:main-periodic-dirichlet}
\begin{align}
\brknmH{u^0 - u_H}
\le C
	\Big( H + \big(\frac{\vep}{\delta} + \delta\big) + \big(\frac{h}{\vep}\big)^2 \Big), \\
\| u^0 - u_H \|_0
\le C
	\Big( H^2 + \big(\frac{\vep}{\delta} + \delta\big) + \big(\frac{h}{\vep}\big)^2 \Big);
\end{align}
\end{subequations}

\item
Under \assref{ass:dirichlet}, if Dirichlet coupling is used, then
\begin{subequations}\label{eq:main-dirichlet}
\begin{align}
\brknmH{u^0 - u_H}
\le C
	\Big( H + \sup_{K_{\delta,i}} \|\overline{\mbA}^0_{K_{\delta,i}}-\mbA^0(\mbx_i) \|_2 + \big(\frac{h}{\vep}\big)^2 \Big), \\
\| u^0 - u_H \|_0
\le C
	\Big( H^2 + \sup_{K_{\delta,i}} \|\overline{\mbA}^0_{K_{\delta,i}}-\mbA^0(\mbx_i) \|_2 + \big(\frac{h}{\vep}\big)^2 \Big).
\end{align}
\end{subequations}
\end{enumerate}
\end{theorem}

\section{Numerical results}\label{sec:numerical-result}

In this section, we have numerical tests for homogeneous finite element scheme.
We need to pay attention to solving the micro problems since it may treats problems with periodic boundary condition.
If the periodic coupling condition is adopted for the micro problems,
the corresponding algebraic system of equations for each micro problem must be constructed
in order to fulfill two important properties, namely,
the periodicity and the zero-integral properties.
Technically, these properties
can be enforced either in the discrete function space or in the formulation for the problem.

In \cite{abdulle2009short}
these two properties are imposed at the level of the problem formulation,
by using a Lagrange multiplier and a constraint matrix.
This approach is quite simple to implement, but the resulting linear
systems are saddle point problems. It requires an additional
consideration to apply an efficient iterative
method for such indefinite problems, and thus a direct method is often
employed as in \cite{abdulle2009short}.

Alternative approaches for the $P_1$--nonconforming quadrilateral finite
element \cite{park-sheen-p1quad} are proposed in \cite{yim-sheen-p1nc-per}.
To handle elliptic problems with periodic boundary condition,
they introduce a subspace of the lowest order nonconforming
quadrilateral finite element which is required to have the periodicity
and iterative-based numerical schemes which deals the zero--integral.
Those approaches assemble semi-definite matrix equations
without additional columns for constraints.
The semi-definiteness is advantageous for the numerical schemes since it derives faster convergence.
Theoretical results of the numerical schemes for the semi-definite matrix equations
show a priori error estimate of the numerical solutions.


In our numerical examples,
we employ one of alternative approaches from \cite{yim-sheen-p1nc-per} to solve micro problems.
We take $\mathfrak{E}^\flat$, a linearly independent basis for
$V^{P_1}_{h,\#}$, as the trial and test function space
and assemble the corresponding linear system
without the zero--integral condition.
The corresponding stiffness matrix is a symmetric positive semi-definite system
whose rank deficiency is 1.
For details about the numerical implementation, see Section 5.2 in \cite{yim-sheen-p1nc-per}.
We will investigate efficiency of the alternative approach for micro problems in \secref{subsec:time-comparison}.
It should be mentioned that we use the 2-point Gauss-Legendre quadrature formula for each coordinate in all numerical examples.

\subsection{Periodic diagonal example}\label{subsec:ex-peri-diag}
The first example is a multiscale elliptic problem
\eqref{eq:multiscale-elliptic} in $\O=(0,1)^2$ whose
coefficient is anisotropically periodic in micro scale: 
$$
\mbA^\vep({\mbx}) =
\begin{pmatrix}
\sqrt{2}+\sin(2 \pi {x_1}/{\vep}) & 0 \\ 0 & \sqrt{2}+\sin(2 \pi {x_2}/{\vep})
\end{pmatrix},\quad
\vep=10^{-3}.
$$
By the homogenization theory, it can be easily shown that the associated homogenized tensor $\mbA^0$ is equal to $I$, the identity tensor.
$f(\mbx)$ is set to satisfy that the associated homogenized elliptic problem has the exact solution $u^0({\mbx}) = \sin(\pi x_1) \sin(\pi x_2)$.
For the sake of simplicity we use the macro and the micro mesh consisting of uniform squares. 
The size parameter $\delta$ of each sampling domain is set to be same as $\vep$. We use periodic coupling for micro problems.

\figref{fig:ex-peri-diag} shows error in energy norm, and in $L^2$-norm,
and the difference between two observable homogenized tensors $\mbA^0$ and $\mbA^0_{K_{\delta,i}}$
in the Frobenius norm.
Note that the matrix $2-$norm for a finite dimensional matrix is equivalent to the Frobenius norm.
The theoretical error estimates \eqref{eq:main-periodic} depend on $H$ as well as $h$.
We can observe that the error is decreasing as $H$ is decreasing,
but there is a critical $H$ value where the error does not decrease anymore for fixed $h$,
as particularly in $L^2$-norm.
Furthermore, in order to observe convergence rate as in \eqref{eq:main-periodic}
we have to consider simultaneous reduction of $H$ and $h$ in different orders,
since the theorem shows the dependency on $H$ and $h$ with their own convergence orders.
For instance, simultaneous reduction of $H$ in the second-order and $h$ in the first-order gives convergence order of 2 in energy norm,
as observed in the table.
The error in $L^2$-norm is similar.
The numerical results confirm \eqref{eq:main-periodic} in \thmref{thm:main}, the main convergence result for periodic cases.

\begin{figure}[!t]
\begin{minipage}{0.6\textwidth}
  \footnotesize
  \centering
  \resizebox{\textwidth}{!}{
  \begin{tabular}{c | c  c  c  c  c }
    \hline \hline
    $H$ & $h/\vep$=1/4 & 1/8 & 1/16 & 1/32 & 1/64 \\ \hline 
    \multicolumn{1}{c}{} & \multicolumn{5}{c}{$\brknmH{u^0-u_H}$} \\ \hline
    1/2 & 1.33E-00 & 1.35E-00 & 1.36E-00 & 1.36E-00 & 1.36E-00 \\
    1/4 & 6.98E-01 & 6.99E-01 & 7.03E-01 & 7.04E-01 & 7.05E-01 \\
    1/8 & 3.75E-01 & 3.55E-01 & 3.54E-01 & 3.55E-01 & 3.55E-01 \\
   1/16 & 2.77E-01 & 1.84E-01 & 1.78E-01 & 1.78E-01 & 1.78E-01 \\
   1/32 & 1.54E-01 & 1.04E-01 & 8.98E-02 & 8.90E-02 & 8.90E-02 \\
   1/64 & 1.93E-01 & 6.79E-02 & 4.66E-02 & 4.46E-02 & 4.45E-02 \\ \hline
    \multicolumn{1}{c}{} & \multicolumn{5}{c}{$\|u^0-u_H\|_0$} \\ \hline
    1/2 & 1.21E-01 & 1.20E-01 & 1.20E-01 & 1.21E-01 & 1.21E-01 \\
    1/4 & 4.58E-02 & 3.22E-02 & 3.04E-02 & 3.04E-02 & 3.04E-02 \\
    1/8 & 3.50E-02 & 1.41E-02 & 8.21E-03 & 7.64E-03 & 7.60E-03 \\
   1/16 & 4.96E-02 & 1.20E-02 & 3.69E-03 & 2.06E-03 & 1.91E-03 \\
   1/32 & 2.88E-02 & 1.25E-02 & 3.20E-03 & 9.31E-04 & 5.16E-04 \\
   1/64 & 4.23E-02 & 1.16E-02 & 3.17E-03 & 8.09E-04 & 2.33E-04 \\ \hline
    \multicolumn{1}{c}{} & \multicolumn{5}{c}{$ \sup_{K_{\delta,i}} \|\mbA^0 - \mbA^0_{K_{\delta,i}} \|_F$} \\ \hline
    1/2 & 1.00E-01 & 3.47E-02 & 9.02E-03 & 2.27E-03 & 5.68E-04 \\
    1/4 & 1.07E-01 & 3.42E-02 & 9.02E-03 & 2.27E-03 & 5.68E-04 \\
    1/8 & 1.04E-01 & 3.44E-02 & 9.02E-03 & 2.27E-03 & 5.68E-04 \\
   1/16 & 1.56E-01 & 3.43E-02 & 9.02E-03 & 2.27E-03 & 5.68E-04 \\
   1/32 & 8.65E-02 & 3.62E-02 & 9.02E-03 & 2.27E-03 & 5.68E-04 \\
   1/64 & 1.44E-01 & 3.37E-02 & 9.02E-03 & 2.27E-03 & 5.68E-04 \\
    \hline \hline
  \end{tabular}
  }
\end{minipage}
\begin{minipage}{0.4\textwidth}
\epsfig{figure=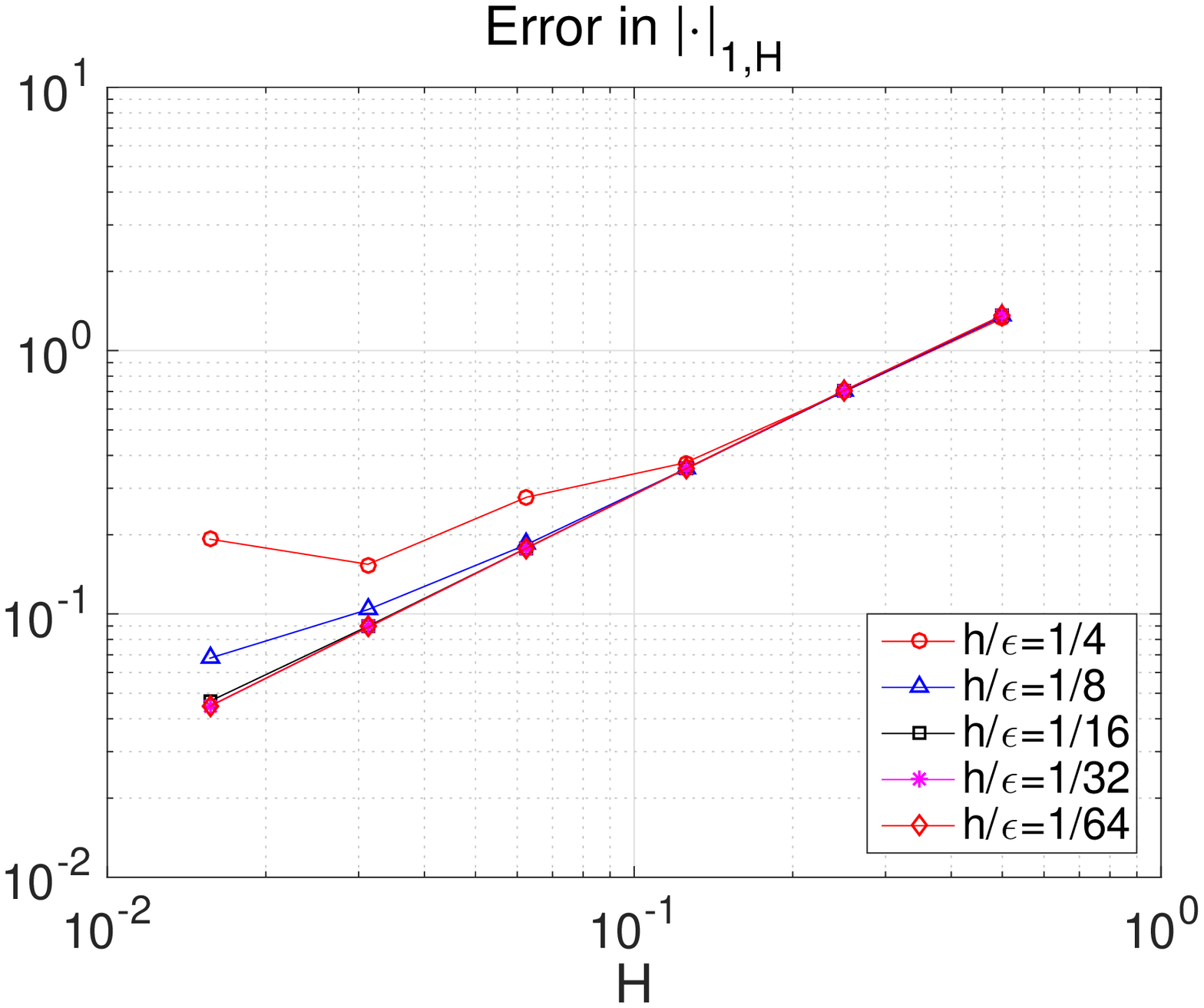,width=\textwidth}
\epsfig{figure=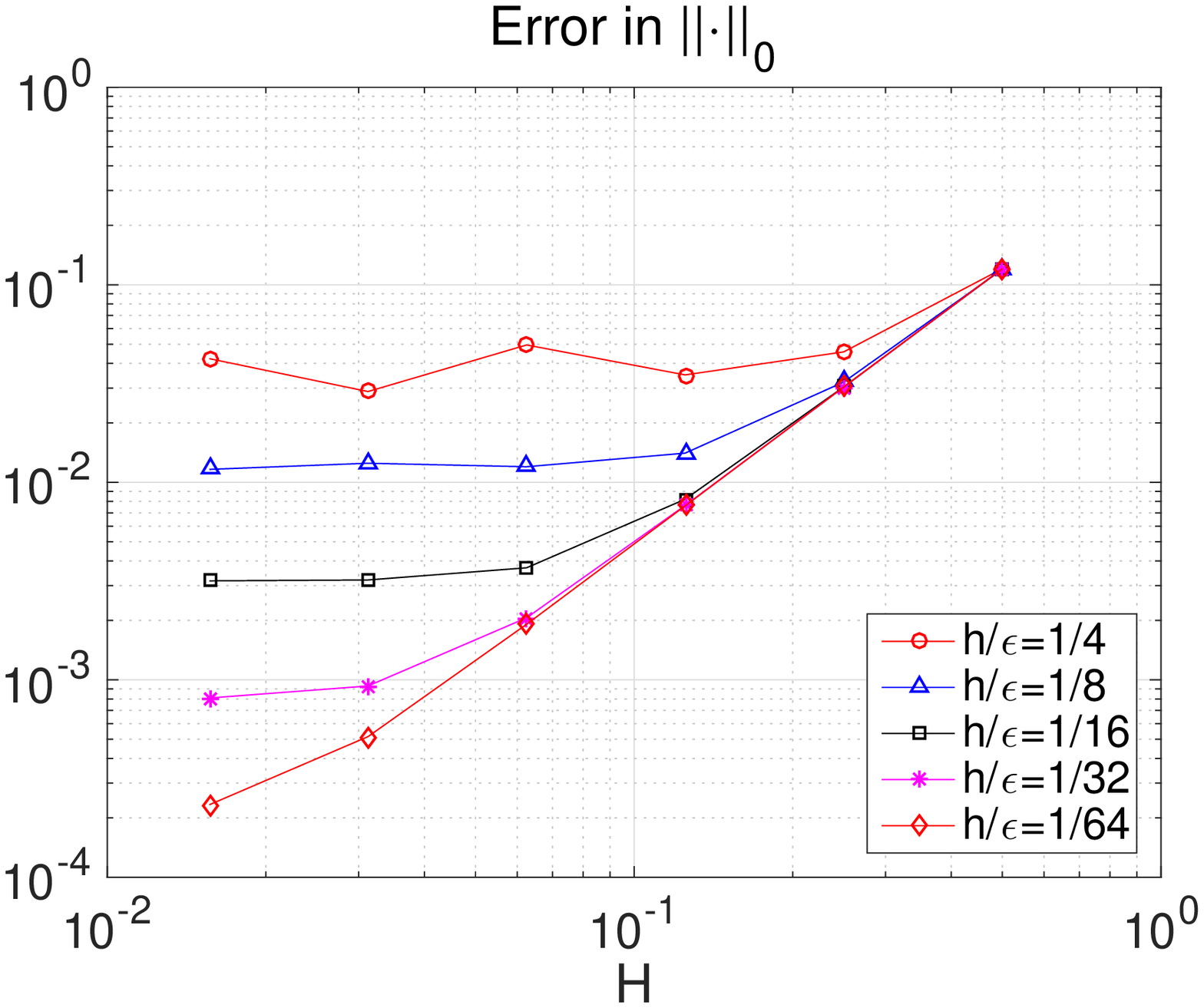,width=\textwidth}
\end{minipage}
\caption{Result of the problem in \secref{subsec:ex-peri-diag}. Errors
  in macro mesh-dependent norm (RT) and  in $L^2$ norm (RB).}
\label{fig:ex-peri-diag}
\end{figure}

\subsubsection{Comparison of elapsed time in solving the micro problem}\label{subsec:time-comparison}
As mentioned in the beginning of this section,
we mainly use the alternative iterative approach based on the Conjugate Gradient method (CG) for micro problems with Dirichlet coupling as well as periodic coupling condition.
Here we investigate the efficiency of the alternative iterative approach over the direct solver
for the periodic coupling case.

We consider three approaches for implementation of the FEHMM scheme.
They only differ in way for setting and solving linear systems corresponding to micro problems.
We describe these approaches in brief.

The first approach uses the conforming $Q_1$ element to assemble a linear system 
for each micro problem. As mentioned in \cite{abdulle2009short},
the assembled system is indefinite due to blocks for constraints.
The number of rows of the system matrix is equal to $n^2+4n+3$, where $n$ is the number
of discretization in each coordinate of each sampling domain.
A direct solver from LAPACK is used to solve the indefinite system numerically.
We name this approach `Direct-$Q_1$'.

The second approach, denoted by `Direct-$P_1$--NC', assembles a linear system using
the $P_1$--nonconforming quadrilateral element in similar manner as the previous approach.
The only difference between two approaches is kind of used finite elements.
Thus the system matrix in this approach is also indefinite, and has the size of $n^2+4n+2$.
This system is solved by the same direct solver as Direct-$Q_1$.

As shown in \tabref{tab:micro-solver-time}, the use of
nonconforming element is more efficient in computing time at least by
10\% than that of
conforming counterpart. We will concentrate on the comparison of
computation times for the nonconforming element between direct and 
iterative linear solvers.

The last approach, denoted by `Iterative-$P_1$--NC' and mainly used throughout the whole
numerical implementations in this paper, also uses the $P_1$--nonconforming quadrilateral element
but in different manner unlike two previous approaches.
This approach uses a basis for the discrete function space with periodic property,
and assembles a corresponding symmetric positive semi-definite system with rank 1 deficiency.
The zero-integral property is imposed as a post-processing procedure.
The size of the system matrix is $n^2+1$, less than previous, due to the absence of constraint blocks.
We solve this semi-definite system in iterative way, by use of the CG.

For the comparison between three approaches,
we again consider the same multiscale elliptic problem in \secref{subsec:ex-peri-diag}.
Each of three approaches is used to solve micro problems numerically, and
(sum of) the elapsed time for micro solver is measured.

\tabref{tab:micro-solver-time} shows the elapsed time for each approach in various combinations of
macro and micro mesh size.
First we can observe that the approach `Direct-$P_1$--NC' takes slightly less than the approaches `Direct-$Q_1$'.
However the elapsed time of `Iterative-$P_1$--NC' approach is much less than other direct approaches.
We can get constant time ratio of the iterative method over other direct methods as the (macro) mesh parameter $H$ varies,
but strongly depending on the micro mesh parameter $h$.
As the finer micro mesh is used, advantage of the iterative method over the elapsed time increases.
Since solving the micro problems is the most time consuming step in the whole numerical test procedure,
it is promising to use the iterative approach which we have adopted.

\begin{table}[t]
  \footnotesize
  \centering
  \begin{tabular}{l | r r r}
    \hline \hline
    & \multicolumn{3}{c}{$h/\vep=1/32$} \\ \cline{2-4}
    $H$ & Direct-$Q_1$ & Direct-$P_1$--NC & Iter-$P_1$--NC \\ \hline
    1/2  & 6.8   & 4.2  & 1.6   \\
    1/4  & 20.3  & 16.4 & 6.3   \\
    1/8  & 73.9  & 66.8 & 25.8  \\
    1/16 & 288.8 & 260.8& 102.5  \\
    \hline 
  \end{tabular}
\vspace{0.3cm}

  \begin{tabular}{l | r r r}
    \hline \hline
 & \multicolumn{3}{c}{$h/\vep=1/64$}  \\ \cline{2-4}
    $H$  & Direct-$Q_1$ & Direct-$P_1$--NC & Iter-$P_1$--NC \\ \hline
    1/2  &  326.0 & 295.5 & 12.8 \\
    1/4  &  1303.3 & 1164.1 & 52.1 \\
    1/8  &  5147.1 & 5143.5 & 213.9 \\
    1/16 &  20943.3 & 18845.1 & 845.5 \\
    \hline 
  \end{tabular}
  \caption{Elapsed time (in second) for micro solvers.}
  \label{tab:micro-solver-time}
\end{table}

\subsection{Periodic example with off-diagonal terms}\label{subsec:ex-peri-off-diag}
In this example we take a tensor whose components are all nonzero with single directional periodicity. 
For $\vep=10^{-3}$,
consider the problem \eqref{eq:multiscale-elliptic} with a multiscale tensor
$$
\mbA^\vep({\mbx}) =
\begin{pmatrix}
\sqrt{2} + \sin(2 \pi {x_1}/{\vep}) &
\frac{1}{2} + \frac{1}{2\sqrt{2}}\sin(2 \pi {x_1}/{\vep})
\\ 
\frac{1}{2} + \frac{1}{2\sqrt{2}}\sin(2 \pi {x_1}/{\vep}) & 
2 + \sin(2 \pi {x_1}/{\vep})
\end{pmatrix},
$$
and the associated homogenized tensor
$
\mbA^0({\mbx}) =
\begin{pmatrix}
1 &
\frac{1}{2\sqrt{2}}
\\ 
\frac{1}{2\sqrt{2}} &
\frac{17-\sqrt{2}}{8}
\end{pmatrix}.
$
We set $f({\mbx})$ to satisfy that the exact homogenized solution $u^0({\mbx}) = \sin(\pi x_1) \sin(\pi x_2)$. 
As the previous example, we use the macro and the micro mesh consisting of uniform squares, 
and $\delta = \vep$ with periodic coupling for each micro problem.

\figref{fig:ex-peri-off-diag} shows that similar results can be obtained in more general periodic case.

\begin{figure}[t]
\begin{minipage}{0.6\textwidth}
  \footnotesize
  \centering
  \resizebox{\textwidth}{!}{
  \begin{tabular}{c | c  c  c  c  c }
    \hline \hline
    $H$ & $h/\vep$=1/4 & 1/8 & 1/16 & 1/32 & 1/64 \\ \hline 
    \multicolumn{1}{c}{} & \multicolumn{5}{c}{$\brknmH{u^0-u_H}$} \\ \hline
    1/2 & 1.35E-00 & 1.36E-00 & 1.36E-00 & 1.36E-00 & 1.36E-00 \\
    1/4 & 6.98E-01 & 7.02E-01 & 7.04E-01 & 7.05E-01 & 7.05E-01 \\
    1/8 & 3.56E-01 & 3.54E-01 & 3.55E-01 & 3.55E-01 & 3.55E-01 \\
   1/16 & 1.96E-01 & 1.78E-01 & 1.78E-01 & 1.78E-01 & 1.78E-01 \\
   1/32 & 1.01E-01 & 9.12E-02 & 8.91E-02 & 8.90E-02 & 8.90E-02 \\
   1/64 & 8.72E-02 & 4.86E-02 & 4.48E-02 & 4.45E-02 & 4.45E-02 \\ \hline
    \multicolumn{1}{c}{} & \multicolumn{5}{c}{$\|u^0-u_H\|_0$} \\ \hline
    1/2 & 1.20E-01 & 1.20E-01 & 1.21E-01 & 1.21E-01 & 1.21E-01 \\
    1/4 & 3.30E-02 & 3.06E-02 & 3.04E-02 & 3.04E-02 & 3.04E-02 \\
    1/8 & 1.54E-02 & 8.83E-03 & 7.68E-03 & 7.60E-03 & 7.60E-03 \\
   1/16 & 2.00E-02 & 4.91E-03 & 2.25E-03 & 1.92E-03 & 1.90E-03 \\
   1/32 & 1.13E-02 & 4.80E-03 & 1.29E-03 & 5.63E-04 & 4.81E-04 \\
   1/64 & 1.68E-02 & 4.45E-03 & 1.21E-03 & 3.25E-04 & 1.41E-04 \\ \hline
    \multicolumn{1}{c}{} & \multicolumn{5}{c}{$ \sup_{K_{\delta,i}} \|\mbA^0 - \mbA^0_{K_{\delta,i}} \|_F$} \\ \hline
    1/2 & 7.99E-02 & 2.76E-02 & 7.17E-03 & 1.80E-03 & 4.52E-04 \\
    1/4 & 8.54E-02 & 2.72E-02 & 7.17E-03 & 1.80E-03 & 4.52E-04 \\
    1/8 & 8.26E-02 & 2.74E-02 & 7.17E-03 & 1.80E-03 & 4.52E-04 \\
   1/16 & 1.24E-01 & 2.73E-02 & 7.17E-03 & 1.80E-03 & 4.52E-04 \\
   1/32 & 6.88E-02 & 2.88E-02 & 7.17E-03 & 1.80E-03 & 4.52E-04 \\
   1/64 & 1.15E-01 & 2.68E-02 & 7.18E-03 & 1.80E-03 & 4.52E-04 \\
    \hline \hline
  \end{tabular}
  }
\end{minipage}
\begin{minipage}{0.4\textwidth}
\epsfig{figure=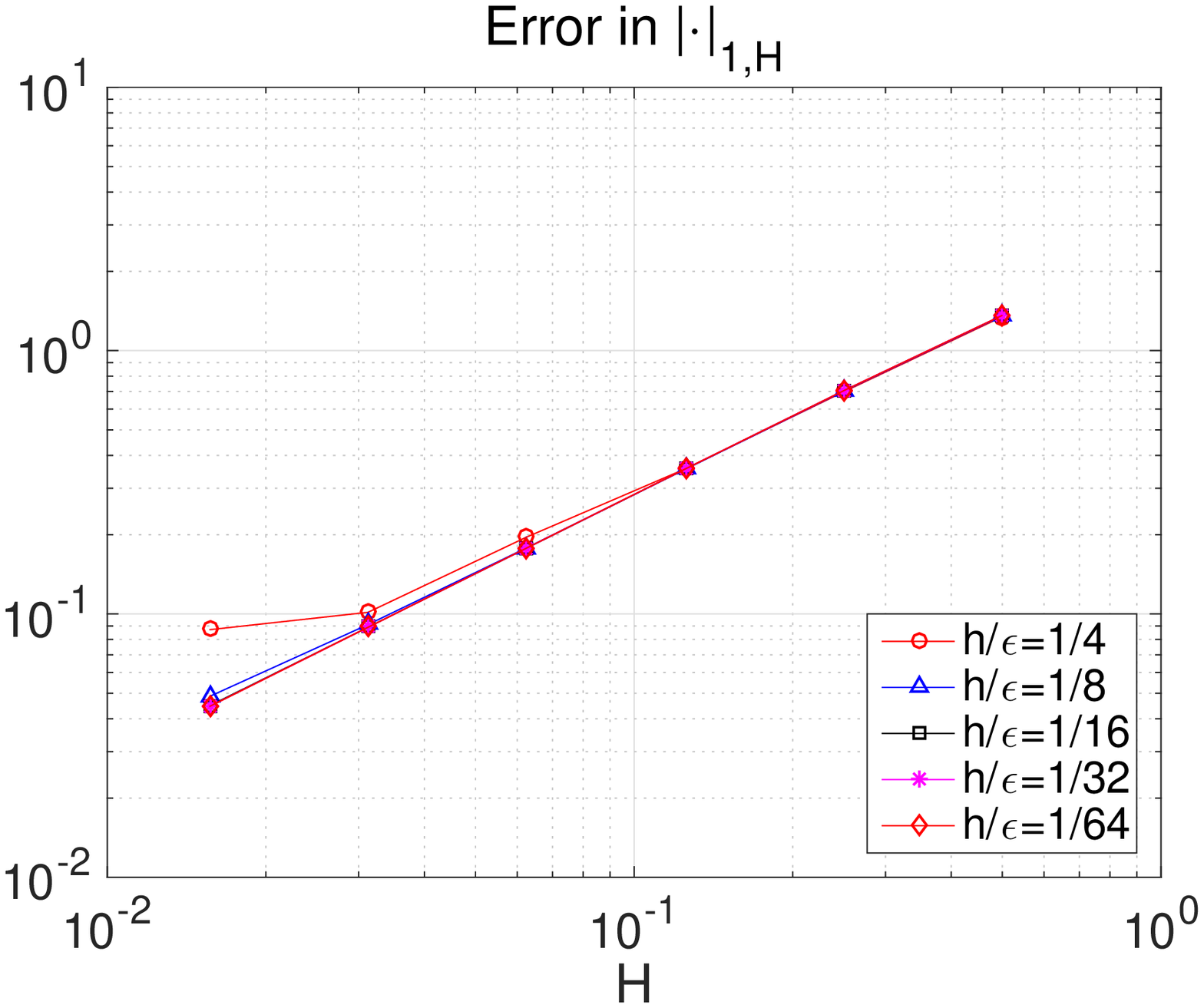,width=\textwidth}
\epsfig{figure=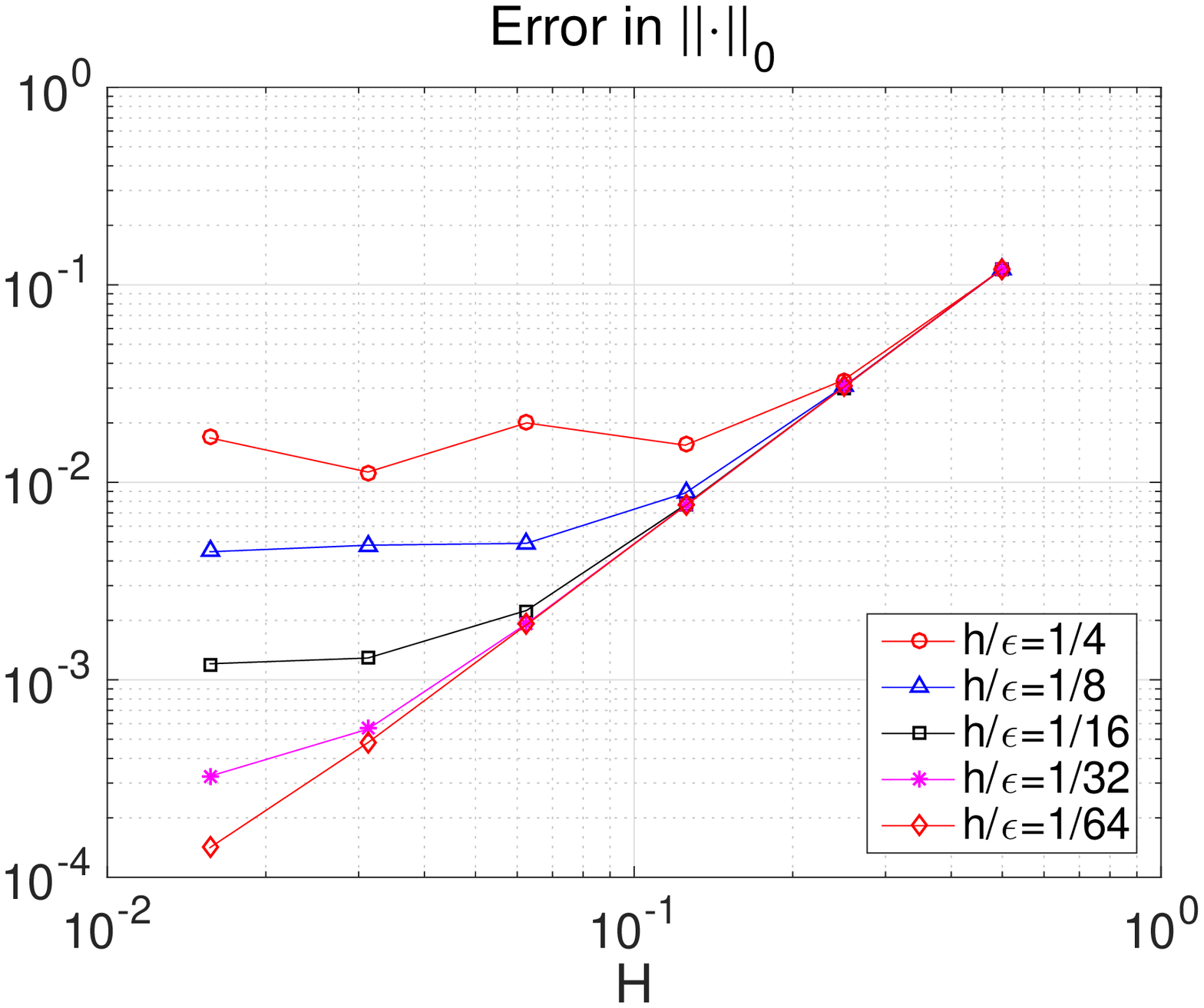,width=\textwidth}
\end{minipage}
\caption{Result of the problem in \secref{subsec:ex-peri-off-diag}.
  Errors in macro mesh-dependent norm (RT) and in $L^2$ norm (RB).}
\label{fig:ex-peri-off-diag}
\end{figure}

\subsection{Example with noninteger-$\vep$-multiple sampling domain and Dirichlet coupling}\label{subsec:ex-diri-noninteger-multiple}
This example, which is originated from \cite{abdulle2009finite}, is to investigate the effect of Dirichlet coupling on micro problems. 
Consider the multiscale elliptic problem with mixed boundary condition
\begin{align*}
-\div \Big( \mbA^{\vep}({\mbx}) \grad u^{\vep}({\mbx}) \Big) &= f({\mbx})
\quad \mbox{in } \O=(0,1)^2, \\
\left.u^{\vep}\right|_{\G_D} &= 0, \\
\left. \nu \cdot \mbA^{\vep} \grad u^{\vep}\right|_{\G_N} &= 0,
\end{align*}
where $\G_D=\{(x_1,x_2) ~|~ x_1=0 \mbox{ or } 1\} \cap \p \O$ 
and $\G_N = \p \O \setminus \G_D$.
We use the multiscale coefficient tensor
$
\mbA^\vep({\mbx}) 
=
\big(2+\cos (2\pi x_1/\vep)\big) I
$
where $\vep = 10^{-3}$, 
the associated homogenized tensor
$
\mbA^0({\mbx}) =
diag(\sqrt{3}, 2)
$,
and $f\equiv 1$ which admits the exact homogenized solution
$
u^0({\mbx}) = -\frac{1}{2\sqrt{3}}x_1(x_1-1)
$.

We use Dirichlet coupling on each micro problem. 
We have three options for sampling domain size $\delta$ which are not multiple of $\vep$; $\delta = 1.1\vep$, $3.1\vep$ and $\sqrt{\vep}$. 
The last option is deduced from \eqref{eq:main-periodic-dirichlet} for the optimal convergence. 
For each macro element 128 micro elements are used, a
sufficiently large number, in order to guarantee that the micro error \eqref{eq:micro-error} can not disrupt the tendency of the total error, except the case of $\sqrt{\vep}$ with 512.

We can observe that error varies depending on size of sampling domains. As shown in \figref{fig:ex-diri-noninteger-multiple},
the bigger size of sampling domains gives the more accurate results.

\begin{figure}[t]
\begin{minipage}{0.6\textwidth}
  \footnotesize
  \centering
  \begin{tabular}{c | c c c }
    \hline \hline
    $H$ & $\delta= 1.1\vep$ (Diri.) & $3.1\vep$ (Diri.) & $\sqrt{\vep}$ (Diri.) \\ \hline 
	\multicolumn{1}{c}{} & \multicolumn{3}{c}{$\brknmH{u^0-u_H}$} \\ \hline
	1/2 & 8.41E-02 & 8.34E-02 & 8.33E-02 \\
	1/4 & 4.22E-02 & 4.17E-02 & 4.17E-02 \\
	1/8 & 2.51E-02 & 2.14E-02 & 2.09E-02 \\
	1/16& 1.50E-02 & 1.11E-02 & 1.04E-02 \\
	1/32& 1.14E-02 & 6.33E-03 & 5.28E-03 \\ \hline
    \multicolumn{1}{c}{} & \multicolumn{3}{c}{$\|u^0-u_H\|_0$} \\ \hline
	1/2 & 1.60E-02 & 1.41E-02 & 1.34E-02 \\
	1/4 & 5.07E-03 & 3.91E-03 & 3.33E-03 \\
	1/8 & 5.11E-03 & 2.29E-03 & 1.20E-03 \\
    1/16 & 3.56E-03 & 1.38E-03 & 3.57E-04 \\
    1/32 & 2.84E-03 & 1.03E-03 & 2.39E-04 \\ \hline
	\multicolumn{1}{c}{} & \multicolumn{3}{c}{$ \sup_{K_{\delta,i}} \|\mbA^0 - \mbA^0_{K_{\delta,i}} \|_F$} \\ \hline
	1/2 & 1.59E-01 & 5.34E-02 & 1.16E-02 \\
	1/4 & 8.45E-02 & 2.97E-02 & 4.82E-03 \\
	1/8 & 1.78E-01 & 6.01E-02 & 1.64E-02 \\
	1/16& 1.42E-01 & 4.79E-02 & 8.22E-03 \\
	1/32& 1.74E-01 & 5.88E-02 & 1.55E-02 \\
    \hline \hline
  \end{tabular}
\end{minipage}
\begin{minipage}{0.35\textwidth}
\epsfig{figure=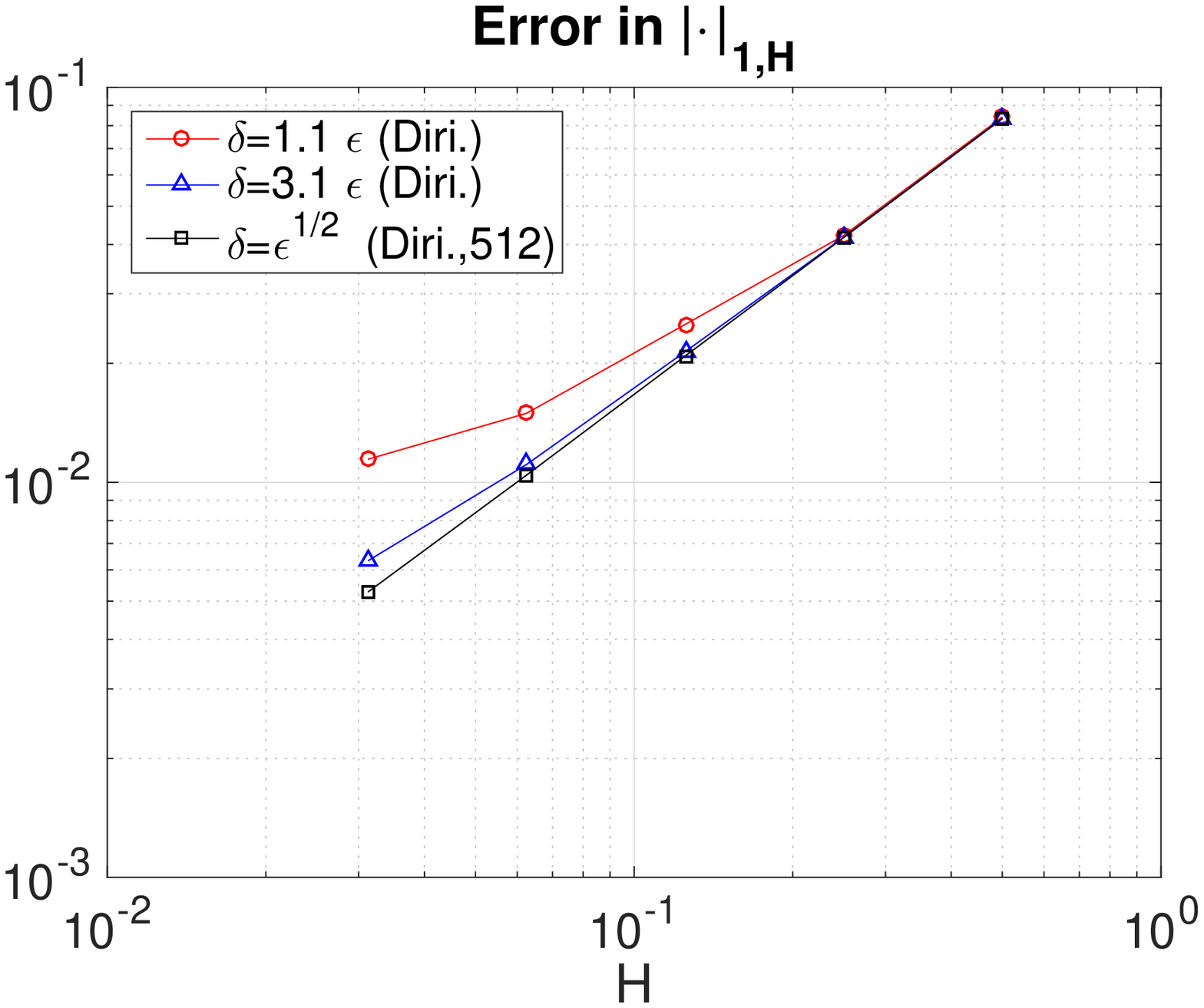,width=\textwidth}
\epsfig{figure=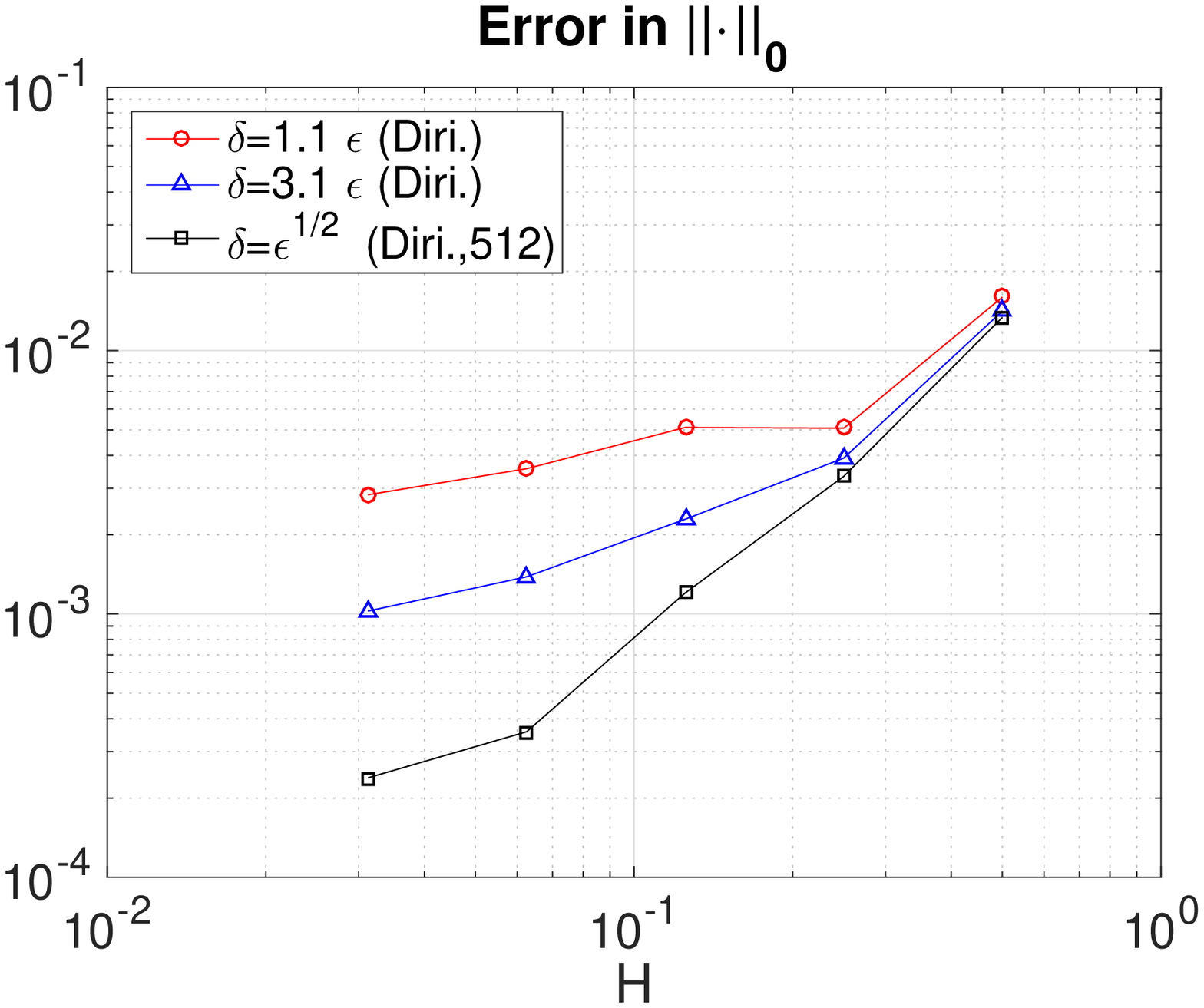,width=\textwidth}
\end{minipage}
 \caption{Result of the problem in \secref{subsec:ex-diri-noninteger-multiple}, $\delta= 1.1\vep$, $3.1\vep$, $\sqrt{\vep}$. (RT) Error in macro mesh-dependent norm and (RB) in $L^2$ norm.}
 \label{fig:ex-diri-noninteger-multiple}
\end{figure}

\subsection{An example on a mixed domain}\label{subsec:ex-mixed-domain}
The last example is a problem on a domain which consists of distinct coefficients.
Let $\O=(0,1)^2$ be decomposed of the two disjoint subdomains
$
\O_1 = \big\{(x_1,x_2) \in \O ~\big|~ x_1 > 0.5 \text{ and } x_2 < 0.5 \big\}
$
and
$
\O_2 = \O \setminus \overline{\O_1}.
$
We consider the second--order elliptic problem with the coefficient tensor
\begin{align*}
\mbA^\vep({\mbx}) =
\begin{pmatrix}
1.1+ \delta_{k,1}\sin(2 \pi {x_1}/{\vep}) & 0 \\ 
0 & 1.1+ \delta_{k,1}\sin(2 \pi {x_1}/{\vep})
\end{pmatrix}
\quad
\text{if}
\quad
\mbx \in \O_k
\end{align*}
where $\vep = 10^{-3}$, and the right hand side $f\equiv 0$. Here $\delta_{ij}$ denotes the standard Kronecker delta.
We impose the homogeneous Neumann boundary condition on the upper and lower boundary,
and the inhomogeneous Dirichlet boundary condition
with values $1$ and $0$ on on the left and right boundaries, respectively.
Any mesh used in this example consists of uniform squares. 
We applied the periodic coupling for micro problems with $\delta = \vep$.
By using the associated homogenized tensor
\begin{align*}
\mbA^0({\mbx}) =
\begin{cases}
\begin{pmatrix}
\sqrt{0.21} & 0 \\ 0 & 1.1
\end{pmatrix} &\qquad \mbox{for } {\mbx} \in \O_1, \\
\begin{pmatrix}
1.1 & 0 \\ 0 & 1.1
\end{pmatrix} &\qquad \mbox{for } {\mbx} \in \O_2,
\end{cases}
\end{align*}
the reference solution $u^0_{ref}$ on $512 \times 512$ mesh is obtained.

Contour plots of the solutions are drawn in \figref{fig:ex-mixed-domain} for comparison.
The plot on top left is for the FEM solution $u^\vep_{ref}$, and
the top right plot is for the FEM solution $u^0_{ref}$ of the homogenized problem.
Both solutions are obtained on $512 \times 512$ uniform square mesh.
The plot on bottom left is for the FEHMM solution $u_H$
from the macro mesh with $8\times 8$ uniform squares,
and the micro mesh with $16\times 16$ uniform squares.
The contour plots show the resemblance of
the FEHMM solution to the solution of the homogenized problem
as well as the solution of the original multiscale problem.
The table shows error of FEHMM solutions to the reference solution in energy norm, and in $L^2$-norm.
We can observe the reduction of error due to decreasing $H$ and $h$,
but not as much as the purely periodic case.

\begin{figure}[t]
 \epsfig{figure=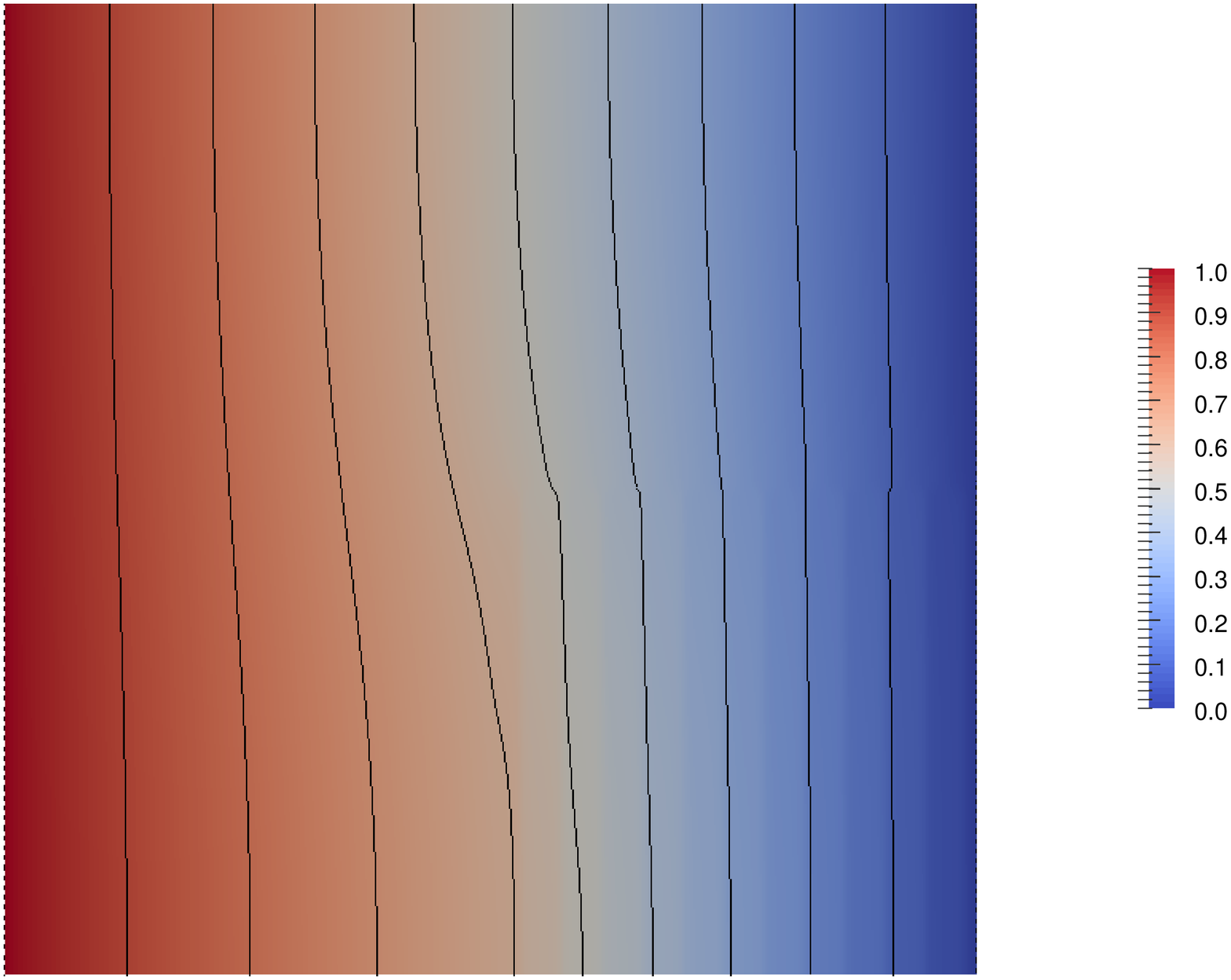,width=0.5\textwidth}
 \epsfig{figure=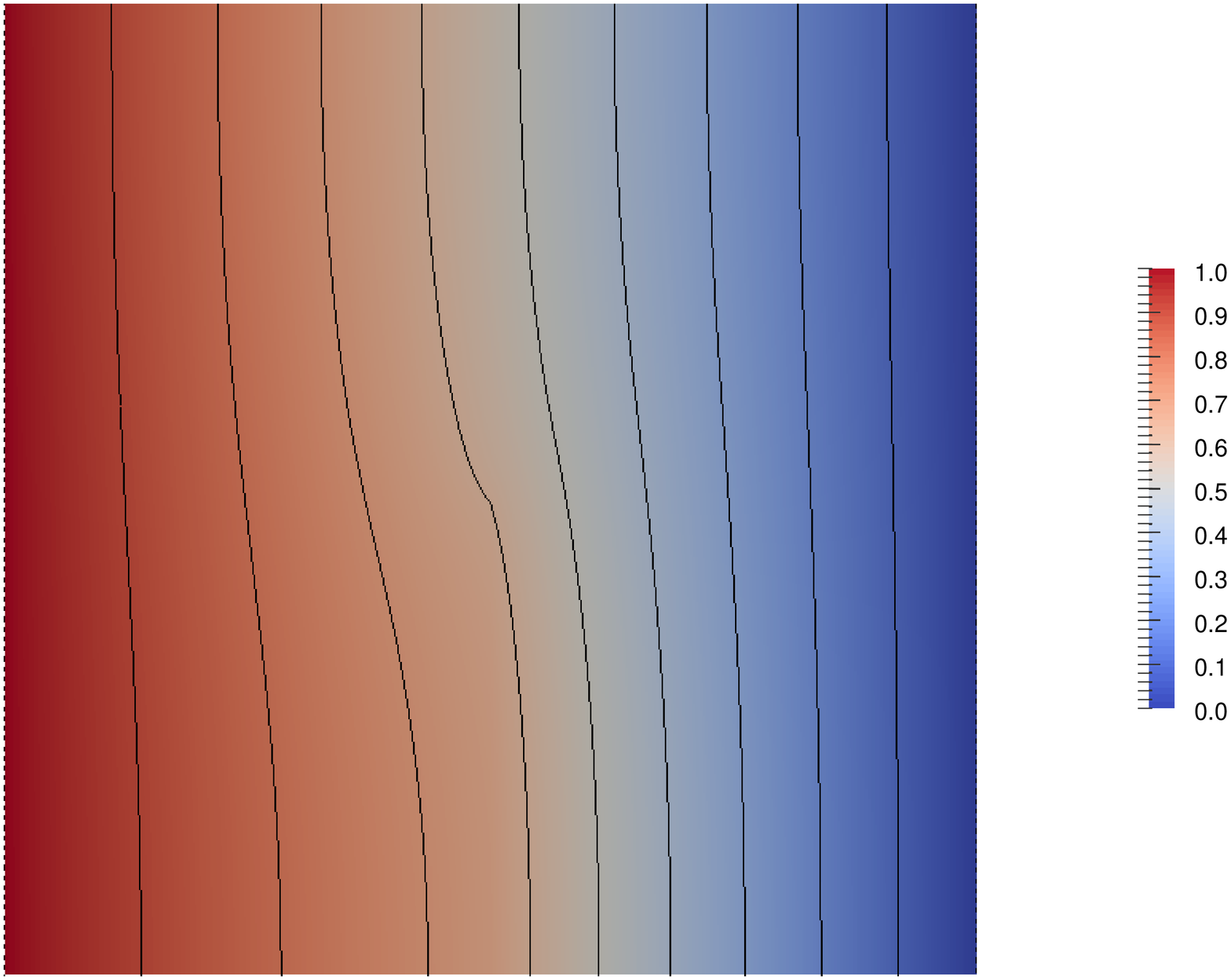,width=0.5\textwidth} \\
 \epsfig{figure=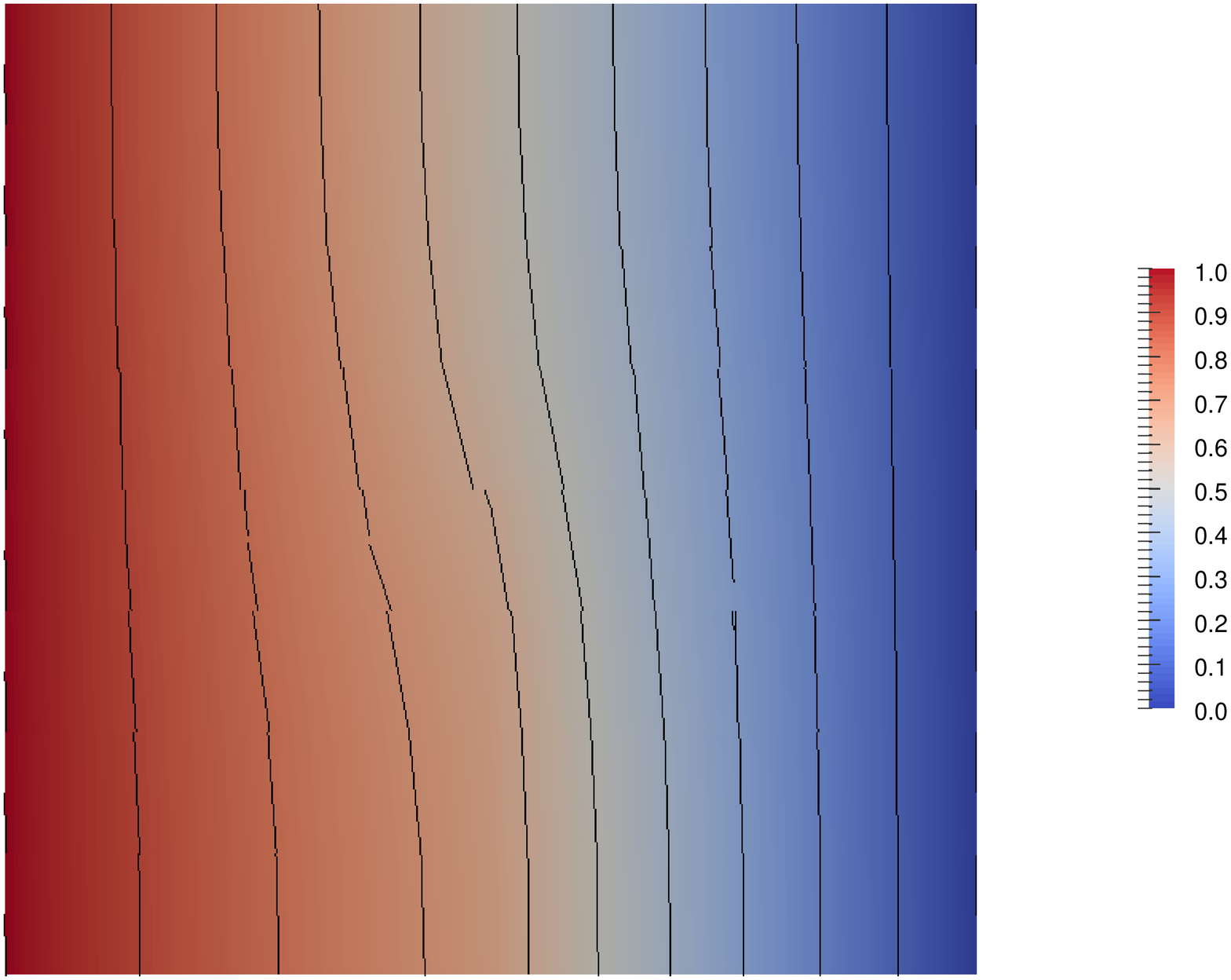,width=0.5\textwidth}
   \qquad
   \footnotesize
   \begin{tabular}[b]{c | c  c  c }
    \hline \hline
    $H$ & $h/\vep$=1/16 & 1/32 & 1/64 \\ \hline 
    \multicolumn{1}{c}{} & \multicolumn{3}{c}{$\|u^0_{ref}-u_H\|_0$} \\ \hline
    1/2 & 9.45E-03 & 9.84E-03 & 9.97E-03 \\
    1/4 & 2.86E-03 & 2.83E-03 & 2.90E-03 \\
    1/8 & 1.53E-03 & 8.31E-04 & 8.20E-04 \\
    1/16& 1.48E-03 & 4.11E-04 & 2.32E-04 \\
    1/32& 1.50E-03 & 3.86E-04 & 1.06E-04 \\
    1/64& 1.58E-03 & 3.91E-04 & 9.73E-05 \\ \hline
    \multicolumn{1}{c}{} & \multicolumn{3}{c}{$\brknmH{u^0_{ref}-u_H}$} \\ \hline
    1/2 & 9.02E-02 & 9.07E-02 & 9.09E-02 \\
    1/4 & 5.32E-02 & 5.34E-02 & 5.35E-02 \\
    1/8 & 3.07E-02 & 3.04E-02 & 3.04E-02 \\
    1/16& 1.78E-02 & 1.69E-02 & 1.69E-02 \\
    1/32& 1.11E-02 & 9.32E-03 & 9.21E-03 \\
    1/64& 8.31E-03 & 5.21E-03 & 4.97E-03 \\
    \hline \hline
  \end{tabular}
\caption{Result of the problem in \secref{subsec:ex-mixed-domain}.
  FEM solution $u^\vep_{ref}$ (TL) and FEM solution $u^0_{ref}$ on
  $512 \times 512$ uniform squares (TR). FEHMM solution $u_H$ using
  $8\times 8$ uniform squares as the macro mesh and $16\times 16$
  uniform squares as the micro mesh (BL).}
\label{fig:ex-mixed-domain}
\end{figure}

\bigskip
\section*{Acknowledgments}
DS was supported in part by
NRF-2017R1A2B3012506 and NRF-2015M3C4A7065662.
\bigskip
\def\cprime{$'$}

\end{document}